\documentclass[11pt]{amsart}

\usepackage{hyperref}
\hypersetup{bookmarksdepth=3}
  
\usepackage{geometry}
\usepackage{amsmath,amssymb}
\usepackage{
mathrsfs}
\usepackage{enumerate}
\usepackage{esint}

\usepackage{tensor}

\usepackage{amsmath}

\usepackage{IEEEtrantools}

\usepackage{dirtytalk}

\usepackage{mathtools}

\usepackage{amssymb,amsmath,amsthm}

\newtheorem{theorem}{Theorem}
\newtheorem{lemma}[theorem]{Lemma}

\newtheorem{definition}[theorem]{Definition}

\theoremstyle{remark}\newtheorem{remark}[theorem]{Remark}

\newtheorem{proposition}[theorem]{Proposition}

\usepackage{graphicx}

\newcommand{\R}{{\mathbb R}}

\newcommand{\Z}{{\mathbb Z}}

\newcommand{\C}{{\mathbb C}}

\newcommand{\T}{{\mathbb T}}

\begin{document}

{\let\thefootnote\relax\footnote{Date: \today. 

\textcopyright 2019 by the authors. Faithful reproduction of this article, in its entirety, by any means is permitted for noncommercial purposes.}}

\title{Unconditional uniqueness of higher order nonlinear Schr\"odinger equations.}


\author{F. Klaus}
\address{Friedrich Klaus, Department of Mathematics, Institute for Analysis, Karlsruhe Institute of Technology (KIT), 76128 Karlsruhe, Germany }
\email{friedrich.klaus@kit.edu}

\author{P. Kunstmann}
\address{Peer Kunstmann, Department of Mathematics, Institute for Analysis, Karlsruhe Institute of Technology (KIT), 76128 Karlsruhe, Germany }
\email{peer.kunstmann@kit.edu}

\author{N. Pattakos}
\address{Nikolaos Pattakos, Department of Mathematics, Institute for Analysis, Karlsruhe Institute of Technology (KIT), 76128 Karlsruhe, Germany }
\email{nikolaos.pattakos@kit.edu}

\begin{abstract}
{We show the existence of weak solutions in the extended sense of the Cauchy problem for the cubic fourth order nonlinear Schr\"odinger equation with initial data $u_{0}\in X$, where $X\in\{M_{2,q}^{s}(\R), H^{\sigma}(\T), H^{s_{1}}(\R)+H^{s_{2}}(\T)\}$ and $q\in[1,2]$, $s\geq0$, or $\sigma\geq0$, or $s_{2}\geq s_{1}\geq0$. Moreover, if $M_{2,q}^{s}(\R)\hookrightarrow L^{3}(\R)$, or if $\sigma\geq\frac16$ or if $s_{1}\geq\frac16$ and $s_{2}>\frac12$ we show that the Cauchy problem is unconditionally wellposed in $X$. Similar results hold true for all higher order nonlinear Schr\"odinger equations and mixed order NLS due to a factorization property of the corresponding phase factors. For the proof we employ the normal form reduction via the differentiation by parts technique and build upon our previous work.}
\end{abstract}

\maketitle
\pagestyle {myheadings}

\begin{section}{introduction and main results}
\markboth{\normalsize F. Klaus, P. Kunstmann and N. Pattakos }{\normalsize  Higher order NLS and differentiation by parts}

We consider the Cauchy problem associated to the cubic fourth order nonlinear Schr\"odinger equation, also known as the cubic biharmonic NLS, given by 
\begin{equation}
\label{maineq}
\begin{cases} i\partial_{t}u-\partial_{x}^{4}u\pm|u|^{2}u=0\\
u(0,x)=u_{0}(x)\\
\end{cases}
\end{equation}
with initial data $u_{0}\in X$ for $X\in\{M_{2,q}^{s}(\R), H^{s}(\T), H^{s_{1}}(\R)+H^{s_{2}}(\T)\}$. The biharmonic NLS provides a canonical model for nonlinear partial differential equations of super-quadratic order. The study of biharmonic NLS goes back to \cite{KAR} and \cite{KARS} where the partial differential equation was introduced to take into account the role of small fourth-order dispersion terms in the propagation of intense laser beams in a bulk medium with Kerr nonlinearity (for applications of higher order NLS, such as sixth and eighth order NLS, see \cite{CKAA}, \cite{KXM},  \cite{SGA} and \cite{YHC}). 

A large amount of work has been devoted to the Cauchy problem \eqref{maineq} with initial data $u_{0}$ in the Sobolev spaces $H^{s}(\R)$ or $H^{s}(\T)$. In the continuous setting solutions $u$ to this problem enjoy mass and energy conservation
\begin{equation}
\label{masscon}
\|u(t,\cdot)\|_{L^{2}(\R)}=\|u_{0}\|_{L^{2}(\R)},
\end{equation}
\begin{equation}
\label{energycon}
E(u(t,\cdot))\coloneqq\frac12\int_{\R}|\Delta u|^{2}\ dx\ \mp\ \frac1{4}\int_{\R}|u|^{4}\ dx=E(u_{0}),
\end{equation}
and it is known that in the mass subcritical cases (with nonlinearity $|u|^{\alpha-1}u$, $\alpha\in[1,1+\frac8{d})$)  the Cauchy problem \eqref{maineq} is globally wellposed in $L^{2}(\R^d)$ via Strichartz type estimates (similar results hold in $H^{2}(\R^d)$ for the energy subcritical cases), see \cite{FIP} as well as \cite{BAKS}, \cite{BL}, \cite{PAU}, \cite{PS} and the references therein. 

In the periodic setting it is known that the Cauchy problem \eqref{maineq} is globally wellposed in $H^{s}(\T)$ for $s>-\frac13$, see \cite{OT} and \cite{OW}, where the proof is done via the short-time Fourier restriction norm method. For more results we refer the interested reader to \cite{ChP}, \cite{CT}, \cite{OW1} and the references therein.

From \cite{BGOR} and \cite{MNRT} it is known that the (semi)-group $S(t)=e^{it\Delta^{2}}$, $t\in\R$, defined as a Fourier multiplier operator with symbol
\begin{equation}
\label{semist}
\mathcal F(S(t))(\xi)\coloneqq e^{it\xi^{4}}
\end{equation}
is not bounded on $M_{p,q}^{s}(\R)$ (for the definition of these modulation spaces see Section \ref{Prelim776}) unless $p=2$ in which case it is an isometry. If in addition we assume that either $q=1$ and $s\geq0$, or $q>1$ and $s>\frac1{q'}$, then the modulation space is a Banach algebra. Hence, for initial data $u_{0}\in M_{2,q}^{s}(\R)$ an easy Banach contraction argument implies that the Cauchy problem \eqref{maineq} is locally wellposed with the solution $u$ being the fixed point of the operator
\begin{equation}
\label{operT}
T(u)\coloneqq S(t)u_{0}\pm i\int_{0}^{t}S(t-\tau)|u|^{2}u\ d\tau
\end{equation} 
in the space $M(R,T)\coloneqq\{u\in C([0,T],M_{2,q}^{s}(\R)): \|u\|\leq R\approx2\|u_{0}\|_{M_{2,q}^{s}}\}$ for $T>0$ sufficiently small. We should also mention that in \cite{CP1} it was shown that $S(t)$ is bounded from $M_{p',q}^{s}(\R)$ into $M_{p,q}^{s}(\R)$ for $p\geq2$ and as a result small data global existence was obtained still in the case that the modulation spaces are Banach algebras.

One of the goals of this paper is to consider similar questions in the case where the modulation space $M_{2,q}^{s}(\R)$ does not belong to the previously mentioned Banach algebra family, i.e. in the case where $q>1$ and $s\in[0,\frac1{q'}]$. This will be achieved with the use of the differentiation by parts technique which was inspired by the periodic case in \cite{GKO} and was used in \cite{CHKP1}, \cite{CHKP2} and \cite{NP} to study similar questions for the cubic NLS in one dimension
\begin{equation}
\label{NLSusual1}
\begin{cases} i\partial_{t}u-\partial_{x}^{2}u\pm|u|^{2}u=0 &,\ (t,x)\in\mathbb R^{2}\\
u(0,x)=u_{0}(x) &,\ x\in\mathbb R.\\
\end{cases}
\end{equation}
Here let us remark that the (semi)-group of \eqref{NLSusual1}, namely the Schr\"odinger operator $e^{it\Delta}$, is bounded on all modulation spaces $M_{p,q}^{s}(\R)$, $p,q\in[1,\infty]$, $s\in\R$ and not only in the special case $p=2$, see again \cite{BGOR}, \cite{MNRT}.

In the periodic setting, as in \cite{BIT} or \cite{GKO}, the differentiation by parts technique transforms the PDE into a countable system of ODEs for the Fourier coefficients of the solution. In the approach described in \cite{CHKP1}, \cite{CHKP2} and \cite{NP} the authors replaced the Fourier coefficients of periodic functions by the isometric decomposition operators, $\Box_{k}$, in order to have a similar localization in the Fourier space. Using these "box" operators for localization yields a unified approach to the periodic and the continuous settings. According to this, a proof using normal form reduction via differentiation by parts for initial data $u_{0}\in H^{s}(\T)$ can be transformed to a proof for initial data $u_{0}\in M_{2,q}^{s}(\R)$, $q\in[1,2]$. This is even possible for initial data in the "tooth problem" space $H^{s}(\R)+H^{s}(\T)$ (we refer to \cite{CHKP2} for the modifications to be made and also for the explanation why we call this a "tooth problem"). The second goal of this paper is to emphasize these relations. Taking this into account we shall only present the proofs of the main Theorems \ref{th1} and \ref{mainyeah} in the case of $u_{0}\in M_{2,q}^{s}(\R)$. A more detailed explanation is given in Remark \ref{explncases}. 

As it was done in \cite{NP}, in order to give a meaning to solutions of the biharmonic NLS in $C([0,T], M_{2,q}^{s}(\R))$ and to the nonlinearity $\mathcal N(u):=u\bar{u}u$ we need the following definitions which first appeared in \cite{C1}, \cite{C2} where power series solutions to the cubic NLS were studied (see also \cite{GUB} for similar considerations for the KdV).

\begin{definition}
\label{def1}
A sequence of Fourier cutoff operators is a sequence of Fourier multiplier operators $\{T_{N}\}_{N\in \mathbb N}$ on $\mathcal S'(\mathbb R)$ with multipliers $m_{N}:\R\to\C$ such that
\begin{itemize}
\item $m_{N}$ has compact support on $\R$ for every $N\in \mathbb N$,
\item $m_{N}$ is uniformly bounded, i.e. $\sup_{x,N}|m_{N}(x)|<\infty$,
\item $\lim_{N\to\infty}m_{N}(x)=1$, for any $x\in\R$. 
\end{itemize}
\end{definition}

\begin{definition}
\label{def2}
Let $u\in C([0,T],M_{2,q}^{s}(\R)).$ We say that $\mathcal N(u)$ exists and is equal to a distribution $w\in\mathcal S'((0,T)\times\R)$ if for every sequence $\{T_{N}\}_{N\in\mathbb N}$ of Fourier cutoff operators we have
\begin{equation}
\label{wknows}
\lim_{N\to\infty}\mathcal N(T_{N}u)=w,
\end{equation}
in the sense of distributions on $(0,T)\times\R$.
\end{definition}

\begin{definition}
\label{def3}
We say that $u\in C([0,T],M_{2,q}^{s}(\R))$ is a weak solution in the extended sense of NLS (\ref{maineq}) if the following are satisfied 
\begin{itemize}
\item $u(0,x)=u_{0}(x)$,
\item the nonlinearity $\mathcal N(u)$ exists in the sense of Definition \ref{def2},
\item $u$ satisfies (\ref{maineq}) in the sense of distributions on $(0,T)\times\R$, where the nonlinearity $\mathcal N(u)=u|u|^{2}$ is interpreted as above.
\end{itemize}
\end{definition}
Our main result which guarantees existence of weak solutions in the extended sense is the following
\begin{theorem}
\label{th1}
Let $1\leq q\leq2$ and $s\geq0.$ For $u_{0}\in M_{2,q}^{s}(\mathbb R)$ there exists a weak solution in the extended sense $u\in C([0,T];M_{2,q}^{s}(\mathbb R))$ of NLS (\ref{maineq}) with initial condition $u_{0}$, where the time $T$ of existence depends only on $\|u_{0}\|_{M_{2,q}^{s}}.$ Moreover, the solution map is locally Lipschitz continuous. 
\end{theorem}

\begin{remark}
The restriction on the range of $q$ appears when estimating the resonant operator $R_{2}^{t}$. See Lemma \ref{lem} in Section \ref{firststeps}.
\end{remark}

The next theorem is about the unconditional wellposedness of NLS (\ref{maineq}) in modulation spaces, that is, uniqueness of solutions in $C([0,T],M_{2,q}^{s}(\R))$ without intersecting with any auxiliary function space (see \cite{TK} where this notion first appeared).
\begin{theorem}
\label{mainyeah}
For $u_{0}\in M_{2,q}^{s}(\R)$, with either $s\geq0$ and $1\leq q\leq\frac32$ or $\frac32<q\leq2$ and $s>\frac23-\frac1{q}$, the solution $u$ with initial condition $u_{0}$ constructed in Theorem \ref{th1} is unique in $C([0,T],M_{2,q}^{s}(\R))$. 
\end{theorem}

\begin{remark}
\label{explncases}
Having stated Theorems \ref{th1} and \ref{mainyeah} for $u_{0}\in M_{2,q}^{s}(\R)$ let us explain what happens in the cases of $H^{s}(\T)$, $s\geq0$ and $H^{s_{1}}(\R)+H^{s_{2}}(\T)$, $s_{2}\geq s_{1}\geq0$. In the latter case we consider functions on $\mathbb T=\R/\mathbb Z$ as periodic functions on $\R$.

In the periodic setting, Definition \ref{def1} remains the same, whereas in Definition \ref{def2} the limit is taken in the sense of distributions on $(0,T)\times\T$. Then the proof follows the calculations presented in the next sections where instead of considering the quantities $\Box_{n}u$ we have the Fourier coefficients of the periodic function $u$, i.e.
\begin{equation}
\label{fourper}
u_{n}(t)\coloneqq\int_{0}^{1}e^{-2\pi i nx}u(t,x)dx,\ n\in\Z.
\end{equation}
In \cite{OW} the authors study \eqref{maineq} in the periodic setting but they are interested in the energy and therefore, they use differentiation by parts for quadrilinear forms whereas we aim for existence of solutions and thus, we have to study trilinear operators. The corresponding tree structures and the estimates for the multilinear expressions in \cite{OW} are different from the one in the present paper. Existence of local solutions is proved in \cite{OW} by the Fourier restriction norm method and even includes negative Sobolev spaces. The fact that the periodic cubic fourth order NLS is unconditionally wellposed in $H^{s}(\T)$ for $s\geq\frac16$ was already observed in \cite{OW} without including the proof.

In the tooth problem space setting, that is of initial data $u_{0}=v_{0}+w_{0}\in H^{s_{1}}(\R)+H^{s_{2}}(\T)$, having the result for the cubic biharmonic periodic NLS with initial data $w_{0}\in H^{s_{2}}(\T)$ at hand (from the previous paragraph), we consider the cubic modified biharmonic NLS given by 
\begin{equation}
\label{nonperpe}
\begin{cases} i\partial_{t}v-\partial_{x}^{4}v\pm G(w,v)=0 &,\ (t,x)\in\R\times\R\\
v(0,x)=v_{0}(x)\in H^{s_{1}}(\R) &,\ x\in\mathbb R\, ,\\
\end{cases}
\end{equation}
where $G(w,v)$ is the nonlinearity
\begin{equation}
\label{nonli}
G(w,v)=|w+v|^{2}(w+v)-|w|^{2}w=|v|^{2}v+v^{2}\bar{w}+w^{2}\bar{v}+2w|v|^{2}+2v|w|^{2}.
\end{equation}
Then Definition \ref{def1} remains the same and Definitions \ref{def2} and \ref{def3} are the same as Definitions 4 and 5 given in \cite{CHKP2}. The proof of the existence of a weak solution in the extended sense $v$ of \eqref{nonperpe} can then be done by modifying the calculations presented in the next sections and combining them with what has been done in \cite{CHKP2}.
\end{remark}

\begin{remark}
\label{unconposed}
Notice that for $q=2$ in Theorem \ref{mainyeah} we obtain that the cubic fourth order NLS is unconditionally wellposed in $H^{s}(\R)$ for $s\geq\frac16$.
\end{remark}

\begin{remark}
\label{sixthhigher23} 
Theorems \ref{th1} and \ref{mainyeah} (and Remarks \ref{explncases}, \ref{unconposed}) remain true for the following mixed order nonlinear Schr\"odinger equations
\begin{equation}
\label{higheryes}
\begin{cases} i\partial_{t}u-\sum_{j=1}^{M}(-1)^{j}\epsilon_{j}\partial_{x}^{2j}u\pm|u|^{2}u=0\\
u(0,x)=u_{0}(x),\\
\end{cases}
\end{equation}
where $M\in\mathbb N$, $\epsilon_{j}\in\R_{\geq0}$ for $j\in\{1,\ldots,M\}$ and $\sum_{j=1}^{M}\epsilon_{j}>0$. This is the case because the phase factors $\Phi_{2j}$, $j\in\{1,\ldots,M\}$ (see \eqref{aaa23} for their definition) enjoy a special factorization property (see Proposition \ref{factoriPhi375}). For a more detailed argument we refer to Section \ref{higherNLS478}. The question whether such a factorization exists had been asked in \cite[Remark 1.5]{GO}. 
\end{remark}

\begin{remark}
A recent preprint, \cite{NKI}, deals with unconditional uniqueness of other dispersive PDE on the multidimensional torus with the use of differentiation by parts in a more abstract framework. It seems not to be clear how to adapt this to the situation of the present paper.

We should also mention \cite{HESO} where the authors use a different approach to unconditional uniqueness of the cubic NLS \eqref{NLSusual1} which applies to various spatial domains. The main idea is to exploit the relation of solutions of the cubic NLS to solutions of the Gross-Pitaevskii hierarchy.

\end{remark}

The paper is organized as follows: In Section \ref{Prelim776} we present the preliminaries and Section \ref{firststeps} contains the first few steps of the iteration process together with the estimates for the first resonant and non-resonant operators that appear. Section \ref{treesandinduction39} describes the tree notation and the induction step and Section \ref{cbgg342} finishes the argument of the proofs of Theorems \ref{th1} and \ref{mainyeah}. Finally, in Section \ref{higherNLS478} we deal with the higher order NLS \eqref{higheryes}.

The following notation will be used throughout the paper: For a number $1\leq p\leq\infty$ we write $p'$ for its dual exponent, that is the number that satisfies $\frac1{p}+\frac1{p'}=1$. For two quantities $A, B$ (they can be functions or numbers) whenever we write $A\lesssim B$ we mean that there is a universal constant $C>0$ such that $A\leq CB$. For a set $A$ we will use $\#(A)$ and $|A|$ to denote its cardinality.

\end{section}

\begin{section}{preliminaries}
\label{Prelim776}

Let us denote by $S(\R)$ the Schwartz class and by $S'(\R)$ the tempered distributions.
\begin{definition}
\label{deffs1}
Let $Q_{0}=[-\frac12, \frac12)$ and its translations $Q_{k}=Q_{0}+k$ for all $k\in\mathbb Z$. Consider a partition of unity $\{\sigma_{k}=\sigma_{0}(\cdot-k)\}_{k\in\mathbb Z}\subset C^{\infty}(\mathbb R)$ satisfying 
\begin{itemize}
\item
$ \exists c > 0: \,
\forall \eta \in Q_{0}: \,
|\sigma_{0}(\eta)| \geq c$,
\item
$
\mbox{supp}(\sigma_{0}) \subseteq \{\xi\in\R:|\xi|<1\}$.
\end{itemize}
Note that this implies $1=\sigma_0(0)=\sigma_k(k)$ for all $k\in\Z$. 
Given a partition of unity as above, we define the isometric decomposition operators (box operators) as
\begin{equation}
\label{iso}
\Box_{k} := \mathcal F^{(-1)} \sigma_{k} \mathcal F, \quad
\left(\forall k \in \Z \right).
\end{equation}
Then the norm of a tempered distribution $f\in S'(\R)$ in the modulation space $M^{s}_{p,q}(\mathbb R)$, $s\in\mathbb R, 1\leq p,q\leq\infty$, is 
\begin{equation}
\label{def}
\|f\|_{M^{s}_{p,q}}:=\Big\|\Big\{\langle k\rangle^{s}\|\Box_{k}f\|_{L^{p}(\R)}\Big\}_{k\in\Z}\Big\|_{l^{q}(\Z)},
\end{equation}
where we denote by $\langle k\rangle\coloneqq(1+|k|^{2})^{\frac12}$ the Japanese bracket and 
\begin{equation}
\label{defmod62h7}
M_{p,q}^{s}(\R)\coloneqq\Big\{f\in S'(\R):\|f\|_{M_{p,q}^{s}}<\infty\Big\}.
\end{equation}
\end{definition}
It can be proved that different choices of such sequences of functions $\{\sigma_{k}\}_{k\in\mathbb Z}$ lead to equivalent norms in $M^{s}_{p,q}(\mathbb R)$. When $s=0$ we denote the space $M^{0}_{p,q}(\mathbb R)$ by $M_{p,q}(\mathbb R)$. In the special case where $p=q=2$ we have $M_{2,2}^{s}(\R)=H^{s}(\R)$ the usual Sobolev spaces
\begin{equation}
\label{Sobspace}
H^{s}(\R)\coloneqq\Big\{f\in S'(\R)\ :\ \|f\|_{H^{s}(\R)}\coloneqq\Big(\int_{\R}\langle\xi\rangle^{2s}|\hat{f}(\xi)|^{2}d\xi\Big)^{\frac12}<\infty\Big\}.
\end{equation}
In this paper we will use that for $s>1/q'$ and $1\leq p, q\leq\infty$, the embedding 
\begin{equation}
\label{yeye}
M_{p,q}^{s}(\R)\hookrightarrow C_{b}(\R)=\{f:\R\to\C:\ f\ \mbox{continuous and bounded}\},
\end{equation}
and for $\Big(1\leq p_{1}\leq p_{2}\leq \infty$, $1\leq q_{1}\leq q_{2}\leq\infty$, $s_{1}\geq s_{2}\Big)$ or $\Big(1\leq p_{1}\leq p_{2}\leq \infty$, $1\leq q_{2}<q_{1}\leq\infty$, $s_{1}>s_{2}+\frac1{q_{2}}-\frac1{q_{1}}\Big)$ the embedding
\begin{equation}
\label{yeye233}
M_{p_{1}, q_{1}}^{s_{1}}(\R)\hookrightarrow M_{p_{2}, q_{2}}^{s_{2}}(\R),
\end{equation}
are both continuous and can be found in \cite[Proposition $6.8$ and Proposition $6.5$]{FEI}. Also, by \cite{BH} it is known that for any $1<p\leq\infty$ we have the embedding $M_{p,1}(\R)\hookrightarrow L^{p}(\R)\cap L^{\infty}(\R)$ which together with the fact that $M_{2,2}(\R)=L^{2}(\R)$ and interpolation, imply that for any $p\in[2,\infty]$ we have the embedding $M_{p,p'}(\R)\hookrightarrow L^{p}(\R).$ Later in the proof of Theorem \ref{mainyeah} we will use this fact for $p=3$, that is 
\begin{equation}
\label{hhh}
M_{3,\frac32}(\R)\hookrightarrow L^{3}(\R).
\end{equation}

The following facts will be useful in the calculations presented in the next sections.

Firstly, notice that for $S(t)=e^{it\Delta^{2}}$ the biharmonic Schr\"odinger (semi)-group we have the equality:
\begin{equation}
\label{Sch}
\|S(t)f\|_{2}=\|f\|_{2},
\end{equation}
Secondly, we need the multiplier estimate (see \cite[Proposition 1.9]{BH}), known as Bernstein's inequality:

\begin{lemma}
\label{Bern}
Let $1 \leq p \leq \infty$ and $\sigma \in C^{\infty}_{c}(\R)$. Then the multiplier operator $T_\sigma: S(\R) \to S'(\R)$ defined by
\begin{equation*}
(T_\sigma f) = \mathcal F^{-1}(\sigma \cdot \hat{f}), \quad
\forall f \in S(\R)
\end{equation*}
is bounded on $L^p(\R)$ and
\begin{equation*}
\|T_{\sigma}\|_{L^p(\R)\to L^p(\R)} \lesssim\|\check\sigma\|_{L^{1}(\R)}.
\end{equation*}
\end{lemma}
An immediate consequence is that for $1\leq p_{1}\leq p_{2}\leq\infty$ we have
\begin{equation}
\label{Bern1}
\|\Box_{k}f\|_{p_{2}}\lesssim\|\Box_{k}f\|_{p_{1}},
\end{equation}
where the implicit constant is independent of $k$ and the function $f$ (for a proof see e.g. \cite{NP}). 

Lastly, let us recall the following number theoretic fact (see \cite[Theorem 315]{HW}) which is going to be used in the proof of Theorem \ref{th1}. 

\begin{proposition}
Given an integer $m$, let $d(m)$ denote the number of divisors of $m$. Then we have
\begin{equation}
\label{num}
d(m)\lesssim e^{c\frac{\log m}{\log\log m}}=o(m^{\epsilon}),
\end{equation}
for all $\epsilon>0$. 
\end{proposition}

\end{section}

\begin{section}{description of the iteration process}
\label{firststeps}

The proof follows the same steps as in \cite{NP} but the operators that appear from applying the differentiation by parts technique are different and have to be estimated differently in order to control them in the appropriate spaces. For this reason we will be detailed only in those steps where a different approach is needed.

In the space $M_{2,q}^{s}(\R)$ there is a more convenient expression for its norm which is the one we are going to use in our calculations. Let us denote by $\tilde\Box_{k}$ the frequency projection operator $\mathcal F^{(-1)}1_{[k,k+1]}\mathcal F$, where $1_{[k,k+1]}$ is the characteristic function of the interval $[k,k+1]$, $k\in\Z$. It can be proved that 
\begin{equation}
\label{normeq}
\|f\|_{M_{2,q}^{s}}\approx\Big(\sum_{k\in\Z}\langle k\rangle^{sq}\|\tilde\Box_{k}f\|_{2}^{q}\Big)^{\frac1{q}},
\end{equation}
that is, the two norms are equivalent in $M_{2,q}^{s}(\R)$. We are going to use expression (\ref{normeq}) for the norm in $M_{2,q}^{s}(\R)$ and for convenience we will still write $\Box_{n}$ instead of $\tilde\Box_{n}$ and $\sigma_{k}$ instead of $1_{[k,k+1]}$. 

From here on, we consider only the case $s=0$ in Theorem \ref{th1} since for $s>0$ similar considerations apply. See Remark \ref{reme} for a more detailed explanation.

The next notations are essential for the analysis that will follow. For $n\in\mathbb Z$ let us define
\begin{equation}
\label{ww0}
u_{n}(t,x)=\Box_{n} u(t,x),
\end{equation}
\begin{equation}
\label{ww1}
v(t,x)=e^{it\partial_{x}^{4}}u(t,x),
\end{equation}
\begin{equation}
\label{ww}
v_{n}(t,x)=e^{it\partial_{x}^{4}}u_{n}(t,x)=\Box_{n}[(e^{it\partial_{x}^{4}}u(t,x)]=\Box_{n}v(t,x).
\end{equation}
Also for $(\xi,\xi_{1},\xi_{2},\xi_{3})\in\mathbb R^{4}$ we define the function
\begin{equation}
\label{Phi4}
\Phi_{4}(\xi,\xi_{1},\xi_{2},\xi_{3})=\xi^{4}-\xi_{1}^{4}+\xi_{2}^{4}-\xi_{3}^{4},
\end{equation}
which is equal to (see \cite[Lemma $3.1$]{OT})
\begin{equation}
\label{exactexpre}
\Phi_{4}(\xi,\xi_{1},\xi_{2},\xi_{3})=(\xi-\xi_{1})(\xi-\xi_{3})(\xi^{2}+\xi_{1}^{2}+\xi_{2}^{2}+\xi_{3}^{2}+2(\xi_{1}+\xi_{3})^{2})
\end{equation}
if $\xi=\xi_{1}-\xi_{2}+\xi_{3}$. Notice that if we let
\begin{equation}
\label{Phi2}
\Phi_{2}(\xi,\xi_{1},\xi_{2},\xi_{3})=\xi^{2}-\xi_{1}^{2}+\xi_{2}^{2}-\xi_{3}^{2}
\end{equation}
then under the assumption $\xi=\xi_{1}-\xi_{2}+\xi_{3}$, $\Phi_{2}=2(\xi-\xi_{1})(\xi-\xi_{3})$ and the relation holds
\begin{equation}
\label{relofPhis}
|\Phi_{4}|\sim\max\{|\xi|,|\xi_{1}|,|\xi_{2}|,|\xi_{3}|\}^{2}|\xi-\xi_{1}||\xi-\xi_{3}|\gtrsim|\Phi_{2}|^{2}.
\end{equation}
More generally, for $k\in\mathbb Z_{+}$ we define the function
\begin{equation}
\label{aaa23}
\Phi_{2k}(\xi,\xi_{1},\xi_{2},\xi_{3})\coloneqq \xi^{2k}-\xi_{1}^{2k}+\xi_{2}^{2k}-\xi_{3}^{2k}
\end{equation}
which will be studied in more detail in Section \ref{higherNLS478}.

The main equation \eqref{maineq} implies that
\begin{equation}
\label{main2}
i\partial_{t}u_{n}-\partial_{x}^{4}u_{n}\pm\Box_{n}(|u|^{2}u)=0,
\end{equation}
and by using the expansion $u=\sum_{k}\Box_{k}u$ it is immediate that
$$\Box_{n}(u\bar{u}u)=\Box_{n}\sum_{n_{1},n_{2},n_{3}}u_{n_{1}}\bar{u}_{n_{2}}u_{n_{3}}=\sum_{n_{1}-n_{2}+n_{3}\approx n}\Box_{n}[u_{n_{1}}\bar{u}_{n_{2}}u_{n_{3}}],$$
where by $\approx n$ we mean $=n$ or $=n+1$ or $=n-1$. During the calculations we will also write $\xi\approx n$ where $\xi$ is going to be a continuous variable and $n$ an integer. By that we will mean that $\xi\in[n,n+1)$ or more generally that $\xi$ is in a suitable interval around $n$.

Next we do the change of variables $u_{n}(t,x)=e^{-it\partial_{x}^{4}}v_{n}(t,x)$ and arrive at the expression
\begin{equation}
\label{main3}
\partial_{t}v_{n}=\pm i\sum_{n_{1}-n_{2}+n_{3}\approx n}\Box_{n}\Big(e^{it\partial_{x}^{4}}[e^{-it\partial_{x}^{4}}v_{n_{1}}\cdot e^{it\partial_{x}^{4}}\bar{v}_{n_{2}}\cdot e^{-it\partial_{x}^{4}}v_{n_{3}}]\Big).
\end{equation}
The $1$st generation operators are given by
\begin{equation}
\label{main4}
Q^{1,t}_{n}(v_{n_{1}},\bar{v}_{n_{2}},v_{n_{3}})(x)=\Box_{n}\Big(e^{it\partial_{x}^{4}}[e^{-it\partial_{x}^{4}}v_{n_{1}}\cdot e^{it\partial_{x}^{4}}\bar{v}_{n_{2}}\cdot e^{-it\partial_{x}^{4}}v_{n_{3}}]\Big),
\end{equation}
or in other words
\begin{equation}
\label{anothereqforv}
\partial_{t}v_{n}=\pm i\sum_{n_{1}-n_{2}+n_{3}\approx n}Q^{1,t}_{n}(v_{n_{1}},\bar{v}_{n_{2}},v_{n_{3}}).
\end{equation}
Below we describe the first few steps of the iteration procedure known as differentiation by parts technique. We will define many operators, $R_{1}^{t}, R_{2}^{t}, N_{11}^{t}, \tilde{Q}^{1,t}_{n}, N_{21}^{t}, N_{4}^{t}, N_{31}^{t}$ and we need to be able to control all of them in the appropriate norms. This will be done in Lemmata \ref{lem}, \ref{lemle}, \ref{fir}, \ref{fir1}, \ref{fir2} and \ref{fir3}.

To move forward we use the splitting 
\begin{equation}
\label{main11}
\partial_{t}v_{n}=\pm i\sum_{n_{1}-n_{2}+n_{3}\approx n}Q^{1,t}_{n}(v_{n_{1}},\bar{v}_{n_{2}},v_{n_{3}})=\sum_{\substack{n_{1}\approx n\\ or\\ n_{3}\approx n}}\ldots+\sum_{n_{1}\not\approx n\not\approx n_{3}}\ldots
\end{equation}
and we define the resonant operator part
\begin{equation}
\label{main9}
R^{t}_{2}(v)(n)-R^{t}_{1}(v)(n)=\Big(\sum_{n_{1}\approx n}Q^{1,t}_{n}+\sum_{n_{3}\approx n}Q^{1,t}_{n}\Big)-\sum_{\substack{n_{1}\approx n\\ and\\ n_{3}\approx n}}Q^{1,t}_{n}(v_{n_{1}},\bar{v}_{n_{2}},v_{n_{3}}),
\end{equation}
with $R^{t}_{2}$ being equal to the sum of the first two summands and $R^{t}_{1}$ being equal to the last summand, and the non-resonant operator part
\begin{equation}
\label{main10}
N_{1}^{t}(v)(n)=\sum_{n_{1}\not\approx n\not\approx n_{3}}Q^{1,t}_{n}(v_{n_{1}},\bar{v}_{n_{2}},v_{n_{3}}).
\end{equation}
This implies the following expression for our biharmonic NLS (we drop the factor $\pm i$ in front of the sum since it will play no role in our analysis)
\begin{equation}
\label{mainmain}
\partial_{t}v_{n}=R^{t}_{2}(v)(n)-R^{t}_{1}(v)(n)+N_{1}^{t}(v)(n).
\end{equation}
For the non-resonant part $N_{1}^{t}$ we have to split further as
\begin{equation}
\label{main13}
N_{1}^{t}(v)(n)=N_{11}^{t}(v)(n)+N_{12}^{t}(v)(n),
\end{equation}
where 
$$N_{11}^{t}(v)(n)=\sum_{A_{N}(n)}Q^{1,t}_{n}(v_{n_{1}},\bar{v}_{n_{2}},v_{n_{3}}),$$
\begin{equation}
\label{set1}
A_{N}(n)=\{(n_{1},n_{2},n_{3})\in\mathbb Z^3:n_{1}-n_{2}+n_{3}\approx n, n_{1}\not\approx n\not\approx n_{3}, |\Phi_{4}(n,n_{1},n_{2},n_{3})|\leq N\}
\end{equation}
and
\begin{equation}
\label{idid}
A_{N}(n)^{c}=\{(n_{1},n_{2},n_{3})\in\mathbb Z^3:n_{1}-n_{2}+n_{3}\approx n, n_{1}\not\approx n\not\approx n_{3}, |\Phi_{4}(n,n_{1},n_{2},n_{3})|> N\}.
\end{equation}
The number $N>0$ is considered to be large and will be fixed at the end of the proof. 

At the $N_{12}^{t}$ part we have to split even further keeping in mind that we are on $A_{N}(n)^{c}$. We perform all formal calculations assuming that $v$ is a sufficiently smooth solution. Later, we justify these formal computations.
From \eqref{main4} we know that 
$$\mathcal F(Q^{1,t}_{n}(v_{n_{1}},\bar{v}_{n_{2}},v_{n_{3}}))(\xi)=\sigma_{n}(\xi)\int_{\mathbb R^2}e^{it\Phi_{4}(\xi,\xi_{1},\xi-\xi_{1}-\xi_{3},\xi_{3})}\hat{v}_{n_{1}}(\xi_{1})\hat{\bar{v}}_{n_{2}}(\xi-\xi_{1}-\xi_{3})\hat{v}_{n_{3}}(\xi_{3})\ d\xi_{1}d\xi_{3},$$
and by the usual product rule for the derivative we can write the previous integral as the sum of the following expressions 
\begin{equation}
\label{ttr}
\partial_{t}\Big(\sigma_{n}(\xi)\int_{\mathbb R^2}\frac{e^{it\Phi_{4}(\xi,\xi_{1},\xi-\xi_{1}-\xi_{3},\xi_{3})}}{i\Phi_{4}(\xi,\xi_{1},\xi-\xi_{1}-\xi_{3},\xi_{3})}\ \hat{v}_{n_{1}}(\xi_{1})\hat{\bar{v}}_{n_{2}}(\xi-\xi_{1}-\xi_{3})\hat{v}_{n_{3}}(\xi_{3})\ d\xi_{1}d\xi_{3}\Big)-
\end{equation}
$$\sigma_{n}(\xi)\int_{\mathbb R^2}\frac{e^{it\Phi_{4}(\xi,\xi_{1},\xi-\xi_{1}-\xi_{3},\xi_{3})}}{i\Phi_{4}(\xi,\xi_{1},\xi-\xi_{1}-\xi_{3},\xi_{3})}\ \partial_{t}\Big(\hat{v}_{n_{1}}(\xi_{1})\hat{\bar{v}}_{n_{2}}(\xi-\xi_{1}-\xi_{3})\hat{v}_{n_{3}}(\xi_{3})\Big)\ d\xi_{1}d\xi_{3}.$$ 
Hence, we have the splitting 
\begin{equation}
\label{main5}
\mathcal F(Q^{1,t}_{n})=\partial_{t}\mathcal F(\tilde{Q}^{1,t}_{n})-\mathcal F(T^{1,t}_{n})
\end{equation}
or equivalently
\begin{equation}
\label{main6}
Q^{1,t}_{n}(v_{n_{1}},\bar{v}_{n_{2}},v_{n_{3}})=\partial_{t}(\tilde{Q}^{1,t}_{n}(v_{n_{1}},\bar{v}_{n_{2}},v_{n_{3}}))-T^{1,t}_{n}(v_{n_{1}},\bar{v}_{n_{2}},v_{n_{3}}),
\end{equation}
which allows us to write 
\begin{equation}
\label{nex}
N_{12}^{t}(v)(n)=\partial_{t}(N_{21}^{t}(v)(n))+N_{22}^{t}(v)(n),
\end{equation}
where
\begin{equation}
\label{nex1}
N_{21}^{t}(v)(n)=\sum_{A_{N}(n)^{c}}\tilde{Q}^{1,t}_{n}(v_{n_{1}},\bar{v}_{n_{2}},v_{n_{3}}),
\end{equation}
and
\begin{equation}
\label{nex2}
N_{22}^{t}(v)(n)=\sum_{A_{N}(n)^{c}}T_{n}^{1,t}(v_{n_{1}},\bar{v}_{n_{2}},v_{n_{3}}).
\end{equation}
From the definition of $\tilde{Q}^{1,t}_{n}$ we have 
$$\mathcal F(\tilde{Q}^{1,t}_{n}(v_{n_{1}},\bar{v}_{n_{2}},v_{n_{3}}))(\xi)=e^{it\xi^{4}}\sigma_{n}(\xi)\int_{\mathbb R^2}\frac{\hat{u}_{n_{1}}(\xi_{1})\hat{\bar{u}}_{n_{2}}(\xi-\xi_{1}-\xi_{3})\hat{u}_{n_{3}}(\xi_{3})}{\Phi_{4}(\xi,\xi_{1},\xi-\xi_{1}-\xi_{3},\xi_{3})}\ d\xi_{1}d\xi_{3},$$
and we define the operator $R^{1,t}_{n}$ by
\begin{equation}
\label{rr}
\mathcal F(R^{1,t}_{n}(u_{n_{1}},\bar{u}_{n_{2}},u_{n_{3}}))(\xi)=\sigma_{n}(\xi)\int_{\mathbb R^2}\frac{\hat{u}_{n_{1}}(\xi_{1})\hat{\bar{u}}_{n_{2}}(\xi-\xi_{1}-\xi_{3})\hat{u}_{n_{3}}(\xi_{3})}{\Phi_{4}(\xi,\xi_{1},\xi-\xi_{1}-\xi_{3},\xi_{3})}\ d\xi_{1}d\xi_{3},
\end{equation}
or in other words,
\begin{equation}
\label{rr1}
R^{1,t}_{n}(w_{n_{1}},\bar{w}_{n_{2}},w_{n_{3}})(x)=\int_{\mathbb R^3}e^{ix\xi}\sigma_{n}(\xi)\frac{\hat{w}_{n_{1}}(\xi_{1})\hat{\bar{w}}_{n_{2}}(\xi-\xi_{1}-\xi_{3})\hat{w}_{n_{3}}(\xi_{3})}{\Phi_{4}(\xi,\xi_{1},\xi-\xi_{1}-\xi_{3},\xi_{3})}\ d\xi_{1}d\xi_{3}d\xi.
\end{equation}
Writing out the Fourier transforms of the functions inside the integral it is immediate that
\begin{equation}
\label{main7}
R^{1,t}_{n}(w_{n_{1}},\bar{w}_{n_{2}},w_{n_{3}})(x)=\int_{\mathbb R^3}K^{(1)}_{n}(x,x_{1},y,x_{3})w_{n_{1}}(x)\bar{w}_{n_{2}}(y)w_{n_{3}}(x_{3})\ dx_{1}dydx_{3},
\end{equation}
where
$$K^{(1)}_{n}(x,x_{1},y,x_{3})=\int_{\mathbb R^3}e^{i\xi_{1}(x-x_{1})+i\eta(x-y)+i\xi_{3}(x-x_{3})}\ \frac{\sigma_{n}(\xi_{1}+\eta+\xi_{3})}{\Phi_{4}(\xi_{1}+\eta+\xi_{3},\xi_{1},\eta,\xi_{3})}\ d\xi_{1}d\eta d\xi_{3}=$$
$$\mathcal F^{-1}\rho^{(1)}_{n}(x-x_{1},x-y,x-x_{3})$$
and 
$$\rho_{n}^{(1)}(\xi_{1},\eta,\xi_{3})=\frac{\sigma_{n}(\xi_{1}+\eta+\xi_{3})}{\Phi_{4}(\xi_{1}+\eta+\xi_{3},\xi_{1},\eta,\xi_{3})}.$$

For the remaining part $N_{22}^{t}$ we have to make use of equality \eqref{mainmain} depending on whether the derivative falls on $\hat{v}_{n_{1}}$, on $\hat{\bar{v}}_{n_{2}}$ or on $\hat{v}_{n_{3}}$. The expression we obtain is given by
$$N_{22}^{t}(v)(n)=-2i\sum_{A_{N}(n)^{c}}\Big[\tilde{Q}^{1,t}_{n}(R^{t}_{2}(v)(n_{1})-R^{t}_{1}(v)(n_{1}),\bar{v}_{n_{2}},v_{n_{3}})+\tilde{Q}^{1,t}_{n}(N_{1}^{t}(v)(n_{1}),\bar{v}_{n_{2}},v_{n_{3}})\Big]$$
$$-i\sum_{A_{N}(n)^{c}}\Big[\tilde{Q}^{1,t}_{n}(v_{n_{1}},R^{t}_{2}(\bar{v})(n_{2})-R^{t}_{1}(\bar{v})(n_{2}),v_{n_{3}})+\tilde{Q}^{1,t}_{n}(v_{n_{1}},N_{1}^{t}(\bar{v})(n_{2}),v_{n_{3}})\Big]$$
(the number $2$ that appears in front of the first sum is because the expression is symmetric with respect to $v_{n_{1}}$ and $v_{n_{3}}$). Therefore, we can write $N_{22}^{t}$ as a sum
\begin{equation}
\label{patel2}
N_{22}^{t}(v)(n)=N_{4}^{t}(v)(n)+N_{3}^{t}(v)(n),
\end{equation}
where $N_{4}^{t}(v)(n)$ is the sum including the resonant parts $R^{t}_{2}-R^{t}_{1}$.

In order to continue, the non-resonant part $N_{3}^{t}$ needs to be decomposed even further. It consists of $3$ sums depending on where the operator $N_{1}^{t}$ acts. One of them is the following (similar considerations apply for the remaining sums too)
\begin{equation}
\label{newnew1}
\sum_{A_{N}(n)^{c}}\tilde{Q}^{1,t}_{n}(N_{1}^{t}(v)(n_{1}),\bar{v}_{n_{2}},v_{n_{3}}),
\end{equation}
where
$$N_{1}^{t}(v)(n_{1})=\sum_{m_{1}\not\approx n_{1}\not\approx m_{3}}Q^{1,t}_{n_{1}}(v_{m_{1}},\bar{v}_{m_{2}},v_{m_{3}}),$$
and $n_{1}\approx m_{1}-m_{2}+m_{3}$. Here we have to consider new restrictions on the frequencies $(m_{1},m_{2},m_{3},n_{2},n_{3})$ where the "new" triple of frequencies $m_{1},m_{2},m_{3}$ appears as a "child" of the frequency $n_{1}$. Denoting by $\phi_{1}=\Phi_{4}(n,n_{1},n_{2},n_{3})$ and $\phi_{2}=\Phi_{4}(n_{1},m_{1},m_{2},m_{3})$ we define the set
\begin{equation}
\label{setset1}
C_{1}=\{|\phi_{1}+\phi_{2}|\leq 5^{3}|\phi_{1}|^{1-\frac1{100}}\},
\end{equation}
and split the sum in (\ref{newnew1}) as 
\begin{equation}
\label{patel}
\sum_{A_{N}(n)^{c}}\sum_{C_{1}}\ldots+\sum_{A_{N}(n)^{c}}\sum_{C_{1}^{c}}\ldots=N_{31}^{t}(v)(n)+N_{32}^{t}(v)(n).
\end{equation}

For the $N_{32}^{t}$ part we have to apply differentiation by parts again which creates the $2$nd generation operators. Our first $2$nd generation operator $Q^{2,t}_{n}$ consists of three sums 
$$q^{2,t}_{1,n}=\sum_{A_{N}(n)^{c}}\sum_{ C_{1}^{c}}\tilde{Q}^{1,t}_{n}(N_{1}^{t}(v)(n_{1}),\bar{v}_{n_{2}},v_{n_{3}}),$$
$$q^{2,t}_{2,n}=\sum_{A_{N}(n)^{c}}\sum_{ C_{1}^{c}}\tilde{Q}^{1,t}_{n}(v_{n_{1}},\overline{N_{1}^{t}(v)}(n_{2}),v_{n_{3}}),$$
$$q^{2,t}_{3,n}=\sum_{A_{N}(n)^{c}}\sum_{ C_{1}^{c}}\tilde{Q}^{1,t}_{n}(v_{n_{1}},\bar{v}_{n_{2}},N_{1}^{t}(v)(n_{3})).$$
Let us have a look at the first sum $q^{2,t}_{1,n}$ (we treat the other two in a similar manner). Its Fourier transform is equal to 
$$\sum_{A_{N}(n)^{c}}\sum_{ C_{1}^{c}}\sigma_{n}(\xi)\int_{\mathbb R^2}\frac{e^{it\Phi_{4}(\xi,\xi_{1},\xi-\xi_{1}-\xi_{3},\xi_{3})}}{\Phi_{4}(\xi,\xi_{1},\xi-\xi_{1}-\xi_{3},\xi_{3})}\ \mathcal F(N_{1}^{t}(v)(n_{1}))(\xi_{1})\hat{\bar{v}}_{n_{2}}(\xi-\xi_{1}-\xi_{3})\hat{v}_{n_{3}}(\xi_{3})\ d\xi_{1}d\xi_{3},$$
where
$$\mathcal F(N_{1}^{t}(v)(n_{1}))(\xi_{1})$$
equals
$$\sum_{\substack {n_{1}\approx m_{1}-m_{2}+m_{3} \\ m_{1}\not\approx n_{1}\not\approx m_{3}}}\sigma_{n_{1}}(\xi_{1})\int_{\mathbb R^2}e^{it\Phi_{4}(\xi_{1},\xi_{1}',\xi_{1}-\xi_{1}'-\xi_{3}',\xi_{3}')}\hat{v}_{m_{1}}(\xi_{1}')\hat{\bar{v}}_{m_{2}}(\xi_{1}-\xi_{1}'-\xi_{3}')\hat{v}_{m_{3}}(\xi_{3}')\ d\xi_{1}'d\xi_{3}'.$$
Putting everything together and applying differentiation by parts we can write the integrals inside the sums as
$$\partial_{t}\Big(\sigma_{n}(\xi)\int_{\mathbb R^4}\sigma_{n_{1}}(\xi_{1})\frac{e^{-it(\phi_{1}+\phi_{2})}}{\phi_{1}(\phi_{1}+\phi_{2})}\hat{v}_{m_{1}}(\xi_{1}')\hat{\bar{v}}_{m_{2}}(\xi_{1}-\xi_{1}'-\xi_{3}')\hat{v}_{m_{3}}(\xi_{3}')\hat{\bar{v}}_{n_{2}}(\xi-\xi_{1}-\xi_{3})\hat{v}_{n_{3}}(\xi_{3})d\xi_{1}'d\xi_{3}'d\xi_{1}d\xi_{3}\Big)$$
minus 
$$\sigma_{n}(\xi)\int_{\mathbb R^4}\sigma_{n_{1}}(\xi_{1})\frac{e^{-it(\phi_{1}+\phi_{2})}}{\phi_{1}(\phi_{1}+\phi_{2})}\partial_{t}\Big(\hat{v}_{m_{1}}(\xi_{1}')\hat{\bar{v}}_{m_{2}}(\xi_{1}-\xi_{1}'-\xi_{3}')\hat{v}_{m_{3}}(\xi_{3}')\hat{\bar{v}}_{n_{2}}(\xi-\xi_{1}-\xi_{3})\hat{v}_{n_{3}}(\xi_{3})\Big)d\xi_{1}'d\xi_{3}'d\xi_{1}d\xi_{3},$$
where $\phi_{1}=\Phi_{4}(\xi,\xi_{1},\xi-\xi_{1}-\xi_{3},\xi_{3})$ and $\phi_{2}=\Phi_{4}(\xi_{1},\xi_{1}',\xi_{1}-\xi_{1}'-\xi_{3}',\xi_{3}')$. Equivalently,
\begin{equation}
\label{sms}
\mathcal F(q^{2,t}_{1,n})=\partial_{t}(\tilde q^{2,t}_{1,n})-\mathcal F(\tau^{2,t}_{1,n}).
\end{equation}
Thus, by doing the same at the remaining two sums of $Q^{2,t}_{n}$, namely $q^{2,t}_{2,n}, q^{2,t}_{3,n}$, we obtain the splitting 
\begin{equation}
\label{neq11}
\mathcal F(Q^{2,t}_{n})=\partial_{t}\mathcal F(\tilde Q^{2,t}_{n})-\mathcal F(T^{2,t}_{n}).
\end{equation}
These new operators $\tilde q^{2,t}_{i,n}$, $i=1,2,3$, act on the following "type" of sequences
$$\tilde q^{2,t}_{1,n}(v_{m_{1}},\bar{v}_{m_{2}},v_{m_{3}},\bar{v}_{n_{2}},v_{n_{3}}),$$
with $m_{1}-m_{2}+m_{3}\approx n_{1}$ and $n_{1}-n_{2}+n_{3}\approx n$,
$$\tilde q^{2,t}_{2,n}(v_{n_{1}},\bar{v}_{m_{1}},v_{m_{2}},\bar{v}_{m_{3}},v_{n_{3}}),$$
with $m_{1}-m_{2}+m_{3}\approx n_{2}$ and $n_{1}-n_{2}+n_{3}\approx n$, and
$$\tilde q^{2,t}_{3,n}(v_{n_{1}}\bar{v}_{n_{2}},v_{m_{1}},\bar{v}_{m_{2}},v_{m_{3}}),$$
with $m_{1}-m_{2}+m_{3}\approx n_{3}$ and $n_{1}-n_{2}+n_{3}\approx n$. 

At this point let us stop the procedure and present how all these operators can be estimated.

\begin{remark}
\label{remmm}
In the following part of the paper a series of lemmata will be presented. Unless stated otherwise we will always assume that $1\leq q\leq 2$. 
\end{remark}

\begin{lemma}
\label{lem}
For $j=1,2$
$$\|R^{t}_{j}(v)\|_{l^{q}L^{2}}\lesssim\|v\|^{3}_{M_{2,q}},$$
and
$$\|R^{t}_{j}(v)-R^{t}_{j}(w)\|_{l^{q}L^{2}}\lesssim(\|v\|^{2}_{M_{2,q}}+\|w\|^{2}_{M_{2,q}})\|v-w\|_{M_{2,q}}.$$
\end{lemma}
\begin{proof}
It is the same as the one given in \cite[Lemma 10]{NP}. At exactly this point the requirement $1\leq q\leq2$ is essential.
\end{proof}

\begin{lemma}
\label{lemle}
$$\|N_{11}^{t}(v)\|_{l^{q}L^{2}}\lesssim N^{\frac1{2q'}+}\|v\|^{3}_{M_{2,q}},$$
and
$$\|N_{11}^{t}(v)-N_{11}^{t}(w)\|_{l^{q}L^{2}}\lesssim N^{\frac1{2q'}+}(\|v\|^{2}_{M_{2,q}}+\|w\|^{2}_{M_{2,q}})\|v-w\|_{M_{2,q}}.$$
\end{lemma}
\begin{proof}
The proof is similar to \cite[Lemma 11]{NP} but with a small twist.

Obviously, 
$$\|N_{11}^{t}(v)\|_{L^{2}}\leq\sum_{A_{N}(n)}\|Q^{1,t}_{n}(v_{n_{1}},\bar{v}_{n_{2}},v_{n_{3}})\|_{L^{2}},$$
which from (\ref{Sch}), Lemma \ref{Bern} and H\"older's inequality is estimated above by
$$\sum_{A_{N}(n)}\|u_{n_{1}}\bar{u}_{n_{2}}u_{n_{3}}\|_{L^{2}}\leq\sum_{A_{N}(n)}\|u_{n_{1}}\|_{L^{6}}\|u_{n_{2}}\|_{L^{6}}\|u_{n_{3}}\|_{L^{6}}.$$
Here we make use of \eqref{Bern1} and H\"older's inequality in the discrete variable to obtain the upper bound
$$\sum_{A_{N}(n)}\|u_{n_{1}}\|_{L^{2}}\|u_{n_{2}}\|_{L^{2}}\|u_{n_{3}}\|_{L^{2}}\leq\Big(\sum_{A_{N}(n)}1^{q'}\Big)^{\frac1{q'}}\Big(\sum_{A_{N}(n)}\|u_{n_{1}}\|_{L^{2}}^{q}\|u_{n_{2}}\|_{L^{2}}^{q}\|u_{n_{3}}\|_{L^{2}}^{q}\Big)^{\frac1{q}}=$$
$$\Big[\#(A_{N}(n))\Big]^{\frac1{q'}}\Big(\sum_{A_{N}(n)}\|u_{n_{1}}\|_{L^{2}}^{q}\|u_{n_{2}}\|_{L^{2}}^{q}\|u_{n_{3}}\|_{L^{2}}^{q}\Big)^{\frac1{q}}.$$
Observe that from \eqref{relofPhis} we have the inclusion
$$A_{N}(n)\subset\Big\{(n_{1},n_{2},n_{3})\in\mathbb Z^3:n_{1}-n_{2}+n_{3}\approx n, n_{1}\not\approx n\not\approx n_{3}, |\Phi_{2}(n,n_{1},n_{2},n_{3})|\leq N^{\frac12}\Big\}$$
and from the proof of \cite[Lemma 11]{NP} we have that the cardinality of this last set is $o(N^{\frac12+})$. Thus, we have 
$$\|N_{11}^{t}(v)\|_{l^{q}L^{2}}\lesssim N^{\frac1{2q'}+}\Big(\sum_{n\in\mathbb Z}\sum_{A_{N}(n)}\|u_{n_{1}}\|_{L^{2}}^{q}\|u_{n_{2}}\|_{L^{2}}^{q}\|u_{n_{3}}\|_{L^{2}}^{q}\Big)^{\frac1{q}},$$
and this final summation is estimated by applying Young's inequality in $l^{1}(\mathbb Z)$ providing us with the bound ($\|u\|_{M_{2,q}}=\|v\|_{M_{2,q}}$)
$$\|N_{11}^{t}(v)\|_{l^{q}L^{2}}\lesssim N^{\frac1{2q'}+}\|v\|_{M_{2,q}}^{3}.$$
\end{proof}

\begin{lemma}
\label{fir}
\begin{equation}
\|\tilde{Q}^{1,t}_{n}(v_{n_{1}},\bar{v}_{n_{2}},v_{n_{3}})\|_{2}\lesssim\frac{\|v_{n_{1}}\|_{2}\|v_{n_{2}}\|_{2}\|v_{n_{3}}\|_{2}}{|\Phi_{4}(n_{1}-n_{2}+n_{3},n_{1},n_{2},n_{3})|}.
\end{equation}
\end{lemma}
\begin{proof}
It is the same as the one given in \cite[Lemma 12]{NP}. It is simply a duality argument that uses the localization of the Fourier transforms of the functions $v_{n_{1}}, v_{n_{2}}$ and $v_{n_{3}}$. The denominator turns out to be the absolute value of 
$$\Phi_{4}(n_{1}-n_{2}+n_{3},n_{1},-n_{2},n_{3})=\Phi_{4}(n_{1}-n_{2}+n_{3},n_{1},n_{2},n_{3}).$$
\end{proof}

\begin{remark}
\label{expl}
Notice that Lemma \ref{fir} (this observation applies to Lemma \ref{indu}, too) is true for any triple of functions $f,g,h\in M_{2,q}(\R)$ and the only important property is that they are nicely localised on the Fourier side since we consider their box operators $\Box_{n_{1}}f, \Box_{n_{2}}g$ and $\Box_{n_{3}}h.$ Also, the same proof implies that the operator $Q^{1,t}_{n}(v_{n_{1}},\bar{v}_{n_{2}},v_{n_{3}})$ satisfies the estimate
\begin{equation}
\label{imppo}
\|Q^{1,t}_{n}(v_{n_{1}},\bar{v}_{n_{2}},v_{n_{3}})\|_{2}\lesssim\|v_{n_{1}}\|_{2}\|v_{n_{2}}\|_{2}\|v_{n_{3}}\|_{2}.
\end{equation}
These observations will play an important role later on.
\end{remark}

\begin{lemma}
\label{fir1}
$$\|N_{21}^{t}(v)\|_{l^{q}L^{2}}\lesssim N^{\frac1{q'}-1}\|v\|^{3}_{M_{2,q}},$$
and
$$\|N_{21}^{t}(v)-N_{21}^{t}(w)\|_{l^{q}L^{2}}\lesssim N^{\frac1{q'}-1}(\|v\|_{M_{2,q}}^{2}+\|w\|_{M_{2,q}}^{2})\|v-w\|_{M_{2,q}}.$$
\end{lemma}
\begin{proof}
From Lemma \ref{fir} we have
$$\|N_{21}^{t}(v)\|_{2}\leq\sum_{A_{N}(n)^{c}}\|\tilde{Q}^{1,t}_{n}(v_{n_{1}},\bar{v}_{n_{2}},v_{n_{3}})\|_{2}\lesssim\sum_{A_{N}(n)^{c}}\frac{\|v_{n_{1}}\|_{2}\|v_{n_{2}}\|_{2}\|v_{n_{3}}\|_{2}}{|\Phi_{4}(n_{1}-n_{2}+n_{3},n_{1},n_{2},n_{3})|},$$
and by H\"older's inequality the upper bound
\begin{equation}
\label{thkgod}
\Big(\sum_{A_{N}(n)^{c}}\frac1{|\Phi_{4}(n_{1}-n_{2}+n_{3},n_{1},n_{2},n_{3})|^{q'}}\Big)^{\frac1{q'}}\Big(\sum_{A_{N}(n)^{c}}\|v_{n_{1}}\|_{2}^{q}\|v_{n_{2}}\|_{2}^{q}\|v_{n_{3}}\|_{2}^{q}\Big)^{\frac1{q}}\sim
\end{equation}
\begin{equation}
\label{ss562}
\Big(\sum_{A_{N}(n)^{c}}\frac1{(|n-n_{1}||n-n_{3}|)^{q'}n_{max}^{2q'}}\Big)^{\frac1{q'}} \Big(\sum_{A_{N}(n)^{c}}\|v_{n_{1}}\|_{2}^{q}\|v_{n_{2}}\|_{2}^{q}\|v_{n_{3}}\|_{2}^{q}\Big)^{\frac1{q}}
\end{equation}
where \eqref{relofPhis} was used and $n_{max}\coloneqq\max\{|n|,|n_{1}|,|n_{2}|,|n_{3}|\}$. The first sum of \eqref{ss562} behaves like $N^{\frac1{q'}-1}$ and for the remaining part we apply Young's inequality. 
\end{proof}

\begin{lemma}
\label{fir2}
$$\|N_{4}^{t}(v)\|_{l^{q}L^2}\lesssim N^{\frac1{q'}-1}\|v\|_{M_{2,q}}^{5},$$
and
$$\|N_{4}^{t}(v)-N_{4}^{t}(w)\|_{l^{q}L^2}\lesssim N^{\frac1{q'}-1}(\|v\|_{M_{2,q}}^{4}+\|w\|_{M_{2,q}}^{4})\|v-w\|_{M_{2,q}}.$$
\end{lemma}
\begin{proof}
We repeat the proof of Lemma \ref{fir1} and apply Lemma \ref{lem} to the part $R_{2}^{t}(v)(n_{1})-R_{1}^{t}(v)(n_{1})$.
\end{proof}

\begin{lemma}
\label{fir3}
$$\|N_{31}^{t}(v)\|_{l^{q}L^{2}}\lesssim \|v\|_{M_{2,q}}^{5},$$
and
$$\|N_{31}^{t}(v)-N_{31}^{t}(w)\|_{l^{q}L^{2}}\lesssim (\|v\|^{4}_{M_{2,q}}+\|w\|_{M_{2,q}}^{4})\|v-w\|_{M_{2,q}}.$$
\end{lemma}
\begin{proof}
With the use of Lemma \ref{fir}, Remark \ref{expl} and H\"older's inequality we have 
$$\|N_{31}^{t}(v)\|_{2}\leq\sum_{A_{N}(n)^{c}}\sum_{C_{1}}\|\tilde{Q}^{1,t}_{n}(Q^{1,t}_{n_{1}}(v_{m_{1}},\bar{v}_{m_{2}},v_{m_{3}}),\bar{v}_{n_{2}},v_{n_{3}})\|_{2}\lesssim$$
$$\sum_{A_{N}(n)^{c}}\sum_{C_{1}}\frac{\|v_{m_{1}}\|_{2}\|v_{m_{2}}\|_{2}\|v_{m_{3}}\|_{2}\|v_{n_{2}}\|_{2}\|v_{n_{3}}\|_{2}}{|\Phi_{4}(n_{1}-n_{2}+n_{3},n_{1},n_{2},n_{3})|}\leq$$
\begin{equation}
\label{s33jk8}
\Big(\sum_{A_{N}(n)^{c}}\sum_{C_{1}}\frac1{|\phi_{1}|^{q'}}\Big)^{\frac1{q'}}\Big(\sum_{A_{N}(n)^{c}}\sum_{C_{1}}\|v_{m_{1}}\|_{2}^{q}\|v_{m_{2}}\|_{2}^{q}\|v_{m_{3}}\|_{2}^{q}\|v_{n_{2}}\|_{2}^{q}\|v_{n_{3}}\|_{2}^{q}\Big)^{\frac1{q}}.
\end{equation}
By $q'\geq2$, the first sum of \eqref{s33jk8} is controlled by the series
\begin{equation}
\label{snki9}
\Big(\sum_{A_{N}(n)^{c}}\sum_{C_{1}}\frac1{|\phi_{1}|^{2}}\Big)^{\frac1{2}}.
\end{equation}
Observe that by the definition of the set $C_{1}$ in \eqref{setset1} we have that 
$$|\phi_{2}|\coloneqq|\Phi_{4}(n_{1},m_{1},m_{2},m_{3})|\sim|\phi_{1}|.$$ 
Since $|\mu_{j}|\lesssim (n_{max}^{(j)})^{2}$ for $j=1,2$ where
$$n_{max}^{(1)}=\max\{|n|,|n_{1}|,|n_{2}|,|n_{3}|\},\ n_{max}^{(2)}=\max\{|n_{1}|,|m_{1}|,|m_{2}|,|m_{3}|\}$$
by setting $\mu_{1}=\Phi_{2}(n,n_{1},n_{2},n_{3})$, $\mu_{2}=\Phi_{2}(n_{1},m_{1},m_{2},m_{3})$ we may estimate \eqref{snki9} further by the expression 
$$\Big(\sum_{A_{N}(n)^{c}}\sum_{C_{1}}\frac1{|\mu_{1}\mu_{2}|(n_{max}^{(1)}n_{max}^{(2)})^{2}}\Big)^{\frac12}\lesssim\Big(\sum_{A_{N}(n)^{c}}\sum_{C_{1}}\frac1{|\mu_{1}\mu_{2}|^{1+}}\Big)^{\frac12}\lesssim1.$$
Hence, Young's inequality applied to the second sum of \eqref{s33jk8} finishes the proof.
\end{proof}

The following lemma should be compared to Lemma \ref{fir}.
\begin{lemma}
\label{fir34}
\begin{equation}
\|\tilde q^{2,t}_{1,n}(v_{m_{1}},\bar{v}_{m_{2}},v_{m_{3}},\bar{v}_{n_{2}},v_{n_{3}})\|_{2}\lesssim\frac{\|v_{m_{1}}\|_{2}\|v_{m_{2}}\|_{2}\|v_{m_{3}}\|_{2}\|v_{n_{2}}\|_{2}\|v_{n_{3}}\|_{2}}{|\phi_{1}||\phi_{1}+\phi_{2}|},
\end{equation}
where $\phi_{1}=\Phi_{4}(n_{1}-n_{2},n_{3},n_{1},n_{2},n_{3})$ and $\phi_{2}=\Phi_{4}(m_{1}-m_{2}+m_{3},m_{1},m_{2},m_{3})$.
\end{lemma}
\begin{proof}
Similar to that of Lemma \ref{fir} and \cite[Lemma 17]{NP}.
\end{proof}

Having described the first steps of the iteration process it is time to introduce the correct vocabulary in order to be able to express much more complicated operators. This is done in the next section with the use of a suitable tree notation (see \cite{CHKP2} for a more sophisticated version).

\end{section}

\begin{section}{tree notation and induction step}
\label{treesandinduction39}
The trees used here are similar to the ones described in \cite{GKO} and are exactly the same as the ones used in \cite{NP} with the only difference being the phase factors, $\mu_{j}$, described in \cite[Equation 60]{NP} which we replace here by quantities of the form $\Phi_{4}(n_{1}-n_{2}+n_{3},n_{1},n_{2},n_{3})$. Since this is the heart of the argument and since Lemmata \ref{finaal} and \ref{finaal2} have different proofs than the corresponding ones from \cite[Lemmata 22 and 23]{NP}, we describe the whole procedure again.

A tree $T$ is a finite, partially ordered set with the following properties:
\begin{itemize}
\item For any $a_{1}, a_{2}, a_{3}, a_{4}\in T$ if $a_{4}\leq a_{2}\leq a_{1}$ and $a_{4}\leq a_{3}\leq a_{1}$ then $a_{2}\leq a_{3}$ or $a_{3}\leq a_{2}$. 
\item There exists a maximum element $r\in T$, that is $a\leq r$ for all $a\in T$, which is called the root. 
\end{itemize}
We call the elements of $T$ the nodes of the tree and in this content we will say that $b\in T$ is a child of $a\in T$ (or equivalently, that $a$ is the parent of $b$) if $b\leq a, b\neq a$ and for all $c\in T$ such that $b\leq c\leq a$ we have either $b=c$ or $c=a$. 

A node $a\in T$ is called terminal if it has no children. A nonterminal node $a\in T$ is a node with exactly $3$ children $a_{1}$, the left child, $a_{2}$, the middle child, and $a_{3}$, the right child. We define the sets
\begin{equation}
\label{setsetset}
T^{0}=\{\mbox{all nonterminal nodes}\},
\end{equation}
and
\begin{equation}
\label{setsetset1}
T^{\infty}=\{\mbox{all terminal nodes}\}.
\end{equation}
Obviously, $T=T^{0}\cup T^{\infty}$, $T^{0}\cap T^{\infty}=\emptyset$ and if $|T^{0}|=j\in\Z_{+}$ we have $|T|=3j+1$ and $|T^{\infty}|=2j+1.$ We denote the collection of trees with $j$ parental nodes by
\begin{equation}
\label{setsetset2}
T(j)=\{T\ \mbox{is a tree with}\ |T|=3j+1\}.
\end{equation}
Next, we say that a sequence of trees $\{T_{j}\}_{j=1}^{J}$ is a chronicle of $J$ generations if:
\begin{itemize}
\item $T_{j}\in T(j)$ for all $j=1, 2, \ldots, J$.
\item For all $j=1, 2, \ldots, J-1$, the tree $T_{j+1}$ is obtained by changing one of the terminal nodes of $T_{j}$ into a nonterminal node with exactly $3$ children.
\end{itemize}
Let us also denote by $\mathcal I(J)$ the collection of trees of the $J$th generation. It is easily checked by an induction argument that
\begin{equation}
\label{setsetset3}
|\mathcal I(J)|=1\cdot 3\cdot 5\ldots(2J-1)=:(2J-1)!!.
\end{equation}
Given a chronicle $\{T_{j}\}_{j=1}^{J}$ of $J$ generations we refer to $T_{J}$ as an ordered tree of the $J$th generation. We should keep in mind that the notion of ordered trees comes with associated chronicles. It includes not only the shape of the tree but also how it "grew".

Given an ordered tree $T$ we define an index function $n:T\to\Z$ such that
\begin{itemize}
\item $n_{a}\approx n_{a_{1}}-n_{a_{2}}+n_{a_{3}}$ for all $a\in T^{0}$, where $a_{1}, a_{2}, a_{3}$ are the children of $a$,
\item $n_{a}\not\approx n_{a_{1}}$ and $n_{a}\not\approx n_{a_{3}}$, for all $a\in T^{0}$,
\item $|\phi_{1}|:=|\Phi_{4}(n_{r_{1}}-n_{r_{2}}+n_{r_{3}},n_{r_{1}},n_{r_{2}},n_{r_{3}})|>N$, where $r$ is the root of $T$,
\end{itemize}
and we denote the collection of all such index functions by $\mathcal R(T)$. 

Given an ordered tree $T$ with the chronicle $\{T_{j}\}_{j=1}^{J}$ and associated index functions $n\in\mathcal R(T)$, we need to keep track of the generations of frequencies. Fix an $n\in\mathcal R(T)$ and consider the very first tree $T_{1}$. Its nodes are the root $r$ and its children $r_{1}, r_{2}, r_{3}$. We define the first generation of frequencies by 
$$(n^{(1)},n_{1}^{(1)},n_{2}^{(1)},n_{3}^{(1)}):=(n_{r},n_{r_{1}},n_{r_{2}},n_{r_{3}}).$$
From the definition of the index function we have
$$n^{(1)}\approx n_{1}^{(1)}-n_{2}^{(1)}+n_{3}^{(1)},\ n_{1}^{(1)}\not\approx n^{(1)}\not\approx n_{3}^{(1)}.$$
The ordered tree $T_{2}$ of the second generation is obtained from $T_{1}$ by changing one of its terminal nodes $a=r_{k}\in T_{1}^{\infty}$ for some $k\in\{1,2,3\}$ into a nonterminal node. Then, the second generation of frequencies is defined by
$$(n^{(2)},n_{1}^{(2)},n_{2}^{(2)},n_{3}^{(2)}):=(n_{a},n_{a_{1}},n_{a_{2}},n_{a_{3}}).$$
Thus, we have $n^{(2)}=n_{k}^{(1)}$ for some $k\in\{1,2,3\}$ and from the definition of the index function we have
$$n^{(2)}\approx n_{1}^{(2)}-n_{2}^{(2)}+n_{3}^{(2)},\ n_{1}^{(2)}\not\approx n^{(2)}\not\approx n_{3}^{(2)}.$$

After $j-1$ steps, the ordered tree $T_{j}$ of the $j$th generation is obtained from $T_{j-1}$ by changing one of its terminal nodes $a\in T_{j-1}^{\infty}$ into a nonterminal node. Then, the $j$th generation frequencies are defined as 
$$(n^{(j)},n_{1}^{(j)},n_{2}^{(j)},n_{3}^{(j)}):=(n_{a},n_{a_{1}},n_{a_{2}},n_{a_{3}}),$$
and we have $n^{(j)}=n_{k}^{(m)}(=n_{a})$ for some $m\in\{1,2,\ldots,j-1\}$ and $k\in\{1,2,3\}$, since this corresponds to the frequency of some terminal node in $T_{j-1}$. In addition, from the definition of the index function we have
$$n^{(j)}\approx n_{1}^{(j)}-n_{2}^{(j)}+n_{3}^{(j)},\ n_{1}^{(j)}\not\approx n^{(j)}\not\approx n_{3}^{(j)}.$$
Finally, we use $\phi_{j}$ to denote the corresponding phase factor introduced at the $j$th generation. That is,
\begin{equation}
\label{muuu}
\phi_{j}=\Phi_{4}(n_{1}^{(j)}-n_{2}^{(j)}+n_{3}^{(j)},n_{1}^{(j)},n_{2}^{(j)},n_{3}^{(j)}),
\end{equation}
and we also introduce the quantities
\begin{equation}
\label{qqq}
\tilde\phi_{J}=\sum_{j=1}^{J}\phi_{j},\ \hat{\phi}_{J}=\prod_{j=1}^{J}\tilde\phi_{j}.
\end{equation}
Notice that for $\mu_{j}=\Phi_{2}(n_{1}^{(j)}-n_{2}^{(j)}+n_{3}^{(j)},n_{1}^{(j)},n_{2}^{(j)},n_{3}^{(j)})$ we have the relation
\begin{equation}
\label{anotherrelofPhis}
|\phi_{j}|\sim(n_{max}^{(j)})^{2}|\mu_{j}|\gtrsim|\mu_{j}|^{2},
\end{equation}
where we denote $n_{max}^{(j)}\coloneqq\max\{|n^{(j)}|,|n_{1}^{(j)}|,|n_{2}^{(j)}|,|n_{3}^{(j)}|\}$. 

We should keep in mind that everytime we apply differentiation by parts and split into resonant and non-resonant parts, we need to control the new frequencies that arise from this procedure. For this reason we define the sets
\begin{equation}
\label{sesee}
C_{J}:=\{|\tilde\phi_{J+1}|\leq(2J+3)^{3}|\tilde\phi_{J}|^{1-\frac1{100}}\}\cup\{|\tilde\phi_{J+1}|\leq(2J+3)^{3}|\phi_{1}|^{1-\frac1{100}}\}.
\end{equation}

Let us denote by $T_{\alpha}$ all the nodes of the ordered tree $T$ that are descendants of the node $\alpha\in T^{0}$, i.e. $T_{\alpha}=\{\beta\in T:\beta\leq\alpha,\ \beta\neq\alpha\}$.

We also need to define the principal and final "signs" of a node $a\in T$ which are functions from the tree $T$ into the set $\{\pm1\}$:
\begin{equation}
\label{signsign}
\mbox{psgn}(a)=\begin{cases}
+1,\ a\ \mbox{is not the middle child of his parent}\\
+1,\ a=r,\ \mbox{the root node}\\
-1,\ a\ \mbox{is the middle child of his parent}
\end{cases}
\end{equation}
\begin{equation}
\label{signsignsign}
\mbox{fsgn}(a)=\begin{cases}
+1,\ \mbox{psgn}(a)=+1\ \mbox{and}\ a\ \mbox{has an even number of middle predecessors}\\
-1,\ \mbox{psgn}(a)=+1\ \mbox{and}\ a\ \mbox{has an odd number of middle predecessors}\\
-1,\ \mbox{psgn}(a)=-1\ \mbox{and}\ a\ \mbox{has an even number of middle predecessors}\\
+1,\ \mbox{psgn}(a)=-1\ \mbox{and}\ a\ \mbox{has an odd number of middle predecessors},
\end{cases}
\end{equation}
where the root node $r\in T$ is not considered a middle parent. 

On the general $J$th step we will have to deal with $|\mathcal I(J)|$ operators of the $\tilde q^{J,t}_{T^0,\mathbf n}$ "type" each one corresponding to one of the ordered trees of the $J$th generation, $T\in T(J)$, where $\mathbf n$ is an arbitrary fixed index function on $T$. We have the subindices $T^0$ and $\mathbf n$ because each one of these operators has Fourier transform supported on the cubes with centers the frequencies assigned to the nodes that belong to $T^0$. 

The operators $\tilde q^{J,t}_{T^0,\mathbf n}$ are defined through their Fourier transforms as
\begin{equation}
\label{oops}
\mathcal F(\tilde q^{J,t}_{T^0,\mathbf n}(\{w_{n_\beta}\}_{\beta\in T^{\infty}}))(\xi)=e^{-it\xi^{2}}\mathcal F(R^{J,t}_{T^0,\mathbf n}(\{e^{-it\partial_{x}^{2}}w_{n_\beta}\}_{\beta\in T^{\infty}}))(\xi),
\end{equation}
where the operator $R^{J,t}_{T^0,\mathbf n}$ acts on the functions $\{w_{n_\beta}\}_{\beta\in T^{\infty}}$ as
\begin{equation}
\label{oops1}
R^{J,t}_{T^0,\mathbf n}(\{w_{n_\beta}\}_{\beta\in T^{\infty}})(x)=\int_{\R^{2J+1}}K^{(J)}_{T^0}(x,\{x_{\beta}\}_{\beta\in T^{\infty}})\Big[\otimes_{\beta\in T^{\infty}}w_{n_\beta}(x_{\beta})\Big]\ \prod_{\beta\in T^{\infty}} dx_{\beta},
\end{equation}
and the kernel $K^{(J)}_{T^0,\mathbf n}$ is defined as 
\begin{equation}
\label{oopss}
K^{(J)}_{T^0,\mathbf n}(x,\{x_{\beta}\}_{\beta\in T^{\infty}})=\mathcal F^{-1}(\rho^{(J)}_{T^{0},\mathbf n})(\{x-x_{\beta}\}_{\beta\in T^{\infty}}).
\end{equation}
The formula for the function $\rho^{(J)}_{T^0,\mathbf n}$ with ($|T^{\infty}|=2J+1$)-variables, $\xi_{\beta}$, $\beta\in T^{\infty}$, is
\begin{equation}
\label{jc}
\rho^{(J)}_{T^0,\mathbf n}(\{\xi_{\beta}\}_{\beta\in T^{\infty}})=\Big[\prod_{\alpha\in T^0}\sigma_{n_{\alpha}}\Big(\sum_{\beta\in T^{\infty}\cap T_{\alpha}}\mbox{fsgn}(\beta)\ \xi_{\beta}\Big)\Big]\frac1{\hat{\phi}_{T}},
\end{equation}
where we put
\begin{equation}
\label{yeah}
\hat{\phi}_{T}=\prod_{\alpha\in T^0}\tilde\phi_{\alpha},\ \tilde\phi_{\alpha}=\sum_{\beta\in T^{0}\setminus T_{\alpha}}\phi_{\beta},
\end{equation}
and for $\beta\in T^{0}$ we have
\begin{equation}
\label{yyeah}
\phi_{\beta}=\Phi_{4}(\xi_{\beta_{1}}-\xi_{\beta_{2}}+\xi_{\beta_{3}},\xi_{\beta_{1}},\xi_{\beta_{2}},\xi_{\beta_{3}}),
\end{equation}
where we impose the relation $\xi_{\alpha}=\xi_{\alpha_{1}}-\xi_{\alpha_{2}}+\xi_{\alpha_{3}}$ for every $\alpha\in T^{0}$ that appears in the calculations until we reach the terminal nodes of $T^{\infty}$. This is because in the definition of the function $ \rho^{J,t}_{T^0}$ we need the variables "$\xi$" to be assigned only at the terminal nodes of the tree $T$. We use the notation $\phi_{\beta}$ in similarity to $\phi_{j}$ of equation (\ref{muuu}) because this is the "continuous" version of the discrete quantity. In addition, the variables $\xi_{\alpha_{1}}, \xi_{\alpha_{2}}, \xi_{\alpha_{3}}$ that appear in expression \eqref{jc} are such that $\xi_{\alpha_{1}}\approx n_{\alpha_{1}}, \xi_{\alpha_{2}}\approx n_{\alpha_{2}}, \xi_{\alpha_{3}}\approx n_{\alpha_{3}}$ since the functions $\sigma_{n_{\alpha}}$ are supported in such a way. Therefore, $|\hat{\phi}_{T}|\sim|\hat{\phi}_{J}|$. 

For the induction step of our iteration process we need the following lemma which should be compared to Lemmata \ref{fir} and \ref{fir34}.

\begin{lemma}
\label{indu}
\begin{equation}
\|\tilde q^{J,t}_{T^0,\mathbf n}(\{v_{n_\beta}\}_{\beta\in T^{\infty}})\|_{2}\lesssim\Big(\prod_{\beta\in T^{\infty}}\|v_{n_\beta}\|_{2}\Big)\frac1{|\hat{\phi}_{T}|},
\end{equation}
for every tree $T\in T(J)$ and index function $\mathbf n\in\mathcal R(T)$.
\end{lemma}
\begin{proof}
Follows by a duality argument, using \eqref{oops}, \eqref{oops1} and the nice localization of the functions $v_{n_{\beta}}$ on the Fourier side.
\end{proof}

Given an index function $\mathbf n$ and $2J+1$ functions $\{v_{n_\beta}\}_{\beta\in T^{\infty}}$ and $\alpha\in T^{\infty}$ we define the action of the operator $N_{1}^{t}$ (see \eqref{main10}) on the set $\{v_{n_\beta}\}_{\beta\in T^{\infty}}$ to be the same set as before but with the difference that we have substituted the function $v_{n_\alpha}$ by the new function $N_{1}^{t}(v)(n_\alpha)$. We will denote this new set of functions $N_{1}^{t,\alpha}(\{v_{n_\beta}\}_{\beta\in T^{\infty}})$. Similarly, the action of the operator $R_{2}^{t}-R_{1}^{t}$ (see \eqref{main9}) on the set of functions $\{v_{n_\beta}\}_{\beta\in T^{\infty}}$ will be denoted by $(R_{2}^{t,\alpha}-R_{1}^{t,\alpha})(\{v_{n_\beta}\}_{\beta\in T^{\infty}})$. 

The operator of the $J$th step, $J\geq 2$, that we want to estimate is given by the formula
\begin{equation}
\label{fina}
N_{2}^{(J)}(v)(n):=\sum_{T\in T(J-1)}\sum_{\alpha\in T^{\infty}}\sum_{\substack{\mathbf n\in\mathcal R(T)\\ \mathbf n_{r}=n}}\tilde q^{J-1,t}_{T^0}(N_{1}^{t,\alpha}(\{v_{n_\beta}\}_{\beta\in T^{\infty}})).
\end{equation}

In the following keep in mind that from the splitting procedure we are on the sets $A_{N}(n)^{c},C_{1}^{c},\ldots,C_{J-1}^{c}$ and since $|\phi_{1}|>N$ we trivially have for all $j\in\{2,\ldots,J\}$
\begin{equation}
\label{neqcr7}
|\tilde{\phi}_{j}|\gg(2j+3)^{3}\max\{|\tilde{\phi}_{j-1}|^{1-\frac1{100}}, |\phi_{1}|^{1-\frac1{100}}\}>(2j+3)^{3}N^{1-\frac1{100}}.
\end{equation}

Applying differentiation by parts on the Fourier side we obtain the expression
\begin{equation}
\label{fina1}
N_{2}^{(J)}(v)(n)=\partial_{t}(N_{0}^{(J+1)}(v)(n))+N_{r}^{(J+1)}(v)(n)+N^{(J+1)}(v)(n), 
\end{equation}
where
\begin{equation}
\label{fina2}
N_{0}^{(J+1)}(v)(n):=\sum_{T\in T(J)}\sum_{\substack{\mathbf n\in\mathcal R(T)\\ \mathbf n_{r}=n}}\tilde q^{J,t}_{T^0,\mathbf n}(\{v_{n_\beta}\}_{\beta\in T^{\infty}}),
\end{equation}
and
\begin{equation}
\label{fina3}
N_{r}^{(J+1)}(v)(n):=\sum_{T\in T(J)}\sum_{\alpha\in T^{\infty}}\sum_{\substack{\mathbf n\in\mathcal R(T)\\ \mathbf n_{r}=n}}\tilde q^{J,t}_{T^0,\mathbf n}((R^{t,\alpha}_{2}-R^{t,\alpha}_{1})(\{v_{n_{\beta}}\}_{\beta\in T^{\infty}})),
\end{equation}
and
\begin{equation}
\label{fina4}
N^{(J+1)}(v)(n):=\sum_{T\in T(J)}\sum_{\alpha\in T^{\infty}}\sum_{\substack{\mathbf n\in\mathcal R(T)\\ \mathbf n_{r}=n}}\tilde q^{J,t}_{T^0,\mathbf n}(N_{1}^{t,\alpha}(\{v_{n_{\beta}}\}_{\beta\in T^{\infty}})).
\end{equation}
We also split the operator $N^{(J+1)}$ as the sum
\begin{equation}
\label{fina5}
N^{(J+1)}=N_{1}^{(J+1)}+N_{2}^{(J+1)},
\end{equation}
where $N_{1}^{(J+1)}$ is the restriction of $N^{(J+1)}$ onto $C_{J}$ and $N_{2}^{(J+1)}$ onto $C_{J}^{c}$. 

First we estimate the operators $N_{0}^{(J+1)}$ and $N_{r}^{(J+1)}$ by the following

\begin{lemma}
\label{finaal}
$$\|N_{0}^{(J+1)}(v)\|_{l^{q}L^{2}}\lesssim N^{-\frac{J}2(1-\frac1{100})}\|v\|_{M_{2,q}}^{2J+1},$$
and
$$\|N_{0}^{(J+1)}(v)-N_{0}^{(J+1)}(w)\|_{l^{q}L^{2}}\lesssim N^{-\frac{J}2(1-\frac1{100})}(\|v\|_{M_{2,q}}^{2J}+\|w\|_{M_{2,q}}^{2J})\|v-w\|_{M_{2,q}}.$$

$$\|N_{r}^{(J+1)}(v)\|_{l^{q}L^{2}}\lesssim N^{-\frac{J}2(1-\frac1{100})}\|v\|_{M_{2,q}}^{2J+3},$$
and
$$\|N_{r}^{(J+1)}(v)-N_{r}^{(J+1)}(w)\|_{l^{q}L^{2}}\lesssim N^{-\frac{J}2(1-\frac1{100})}(\|v\|_{M_{2,q}}^{2J+2}+\|w\|_{M_{2,q}}^{2J+2})\|v-w\|_{M_{2,q}}.$$
\end{lemma}
\begin{proof}
From \eqref{yeah} we have that $|\phi_{j}|\lesssim\max\{|\tilde{\phi}_{j-1}|, |\tilde{\phi}_{j}|\}$ which together with \eqref{neqcr7} implies  
\begin{equation}
\label{godhm}
(2j)^{3}N^{1-\frac1{100}}|\phi_{j}|\ll |\tilde{\phi}_{j-1}||\tilde{\phi}_{j}|,\quad \forall j\in\{2,\ldots,J\}.
\end{equation}
This together with the use of \eqref{neqcr7} again shows that
\begin{equation}
\label{g490}
\prod_{j=1}^{J}\Big[(2j+3)^{3}N^{1-\frac1{100}}|\phi_{j}|\Big]\ll |\phi_{1}||\tilde{\phi}_{J}|\prod_{j=2}^{J}\Big[(2j)^{3}N^{1-\frac1{100}}|\phi_{j}|\Big]\ll \prod_{j=1}^{J}|\tilde{\phi}_{j}|^{2}.
\end{equation}
Recalling that 
\begin{equation}
\label{er5376}
\frac1{|\phi_{j}|}\sim\frac1{|\mu_{j}|(n_{max}^{(j)})^{2}|}\lesssim\frac1{|\mu_{j}|^{1+}}
\end{equation}
we obtain
\begin{equation}
\label{nvbt6}
\sum_{\substack{\mathbf n\in\mathcal R(T)\\ \mathbf n_{r}=n}}\prod_{j=1}^{J}\frac1{|\tilde{\phi}_{j}|^{2}}\lesssim\frac{N^{-J(1-\frac1{100})}}{\prod_{j=1}^{J}(2j+3)^{3}}\sum_{\substack{\mathbf n\in\mathcal R(T)\\ \mathbf n_{r}=n}}\prod_{j=1}^{J}\frac1{|\phi_{j}|}\lesssim\frac{N^{-J(1-\frac1{100})}}{\prod_{j=1}^{J}(2j+3)^{3}}\sum_{\substack{\mathbf n\in\mathcal R(T)\\ \mathbf n_{r}=n}}\prod_{j=1}^{J}\frac1{|\mu_{j}|^{1+}}
\end{equation}
where the last expression is bounded from above by 
\begin{equation}
\label{bvv89}
\frac{C^{J}N^{-J(1-\frac1{100})}}{\prod_{j=1}^{J}(2j+3)^{3}}
\end{equation}
for some constant $C>0$.

By Lemma \ref{indu} we obtain
\begin{equation}
\label{d325v}
\sum_{\substack{\mathbf n\in\mathcal R(T)\\ \mathbf n_{r}=n}}\|\tilde q^{J,t}_{T^0,\mathbf n}(\{v_{\beta}\}_{\beta\in T^\infty})\|_{2}\lesssim\sum_{\substack{\mathbf n\in\mathcal R(T)\\ \mathbf n_{r}=n}}\Big(\prod_{\beta\in T^{\infty}}\|v_{n_\beta}\|_{2}\Big)\Big(\prod_{k=1}^{J}\frac1{|\tilde\phi_{k}|}\Big)
\end{equation}
which by H\"older's inequality and \eqref{nvbt6}, \eqref{bvv89} is controlled by
$$\Big(\sum_{\substack{|\phi_{1}|>N\\ |\tilde\phi_{j}|>(2j+1)^{3}N^{1-\frac1{100}}\\ j=2,\ldots,J}}\prod_{j=1}^{J}\frac1{|\tilde{\phi}_{j}|^{q'}}\Big)^{\frac1{q'}}\Big(\sum_{\substack{\mathbf n\in\mathcal R(T)\\ \mathbf n_{r}=n}}\prod_{\beta\in T^{\infty}}\|v_{n_\beta}\|_{2}^{q}\Big)^{\frac1{q}}\leq$$
$$\Big(\sum_{\substack{|\phi_{1}|>N\\ |\tilde\phi_{j}|>(2j+1)^{3}N^{1-\frac1{100}}\\ j=2,\ldots,J}}\prod_{j=1}^{J}\frac1{|\tilde{\phi}_{j}|^{2}}\Big)^{\frac1{2}}\Big(\sum_{\substack{\mathbf n\in\mathcal R(T)\\ \mathbf n_{r}=n}}\prod_{\beta\in T^{\infty}}\|v_{n_\beta}\|_{2}^{q}\Big)^{\frac1{q}}\lesssim$$
\begin{equation}
\label{wow4352}
\frac{C^{\frac{J}2}N^{-\frac{J}2(1-\frac1{100})}}{\prod_{j=1}^{J}(2j+3)^{\frac32}}\Big(\sum_{\substack{\mathbf n\in\mathcal R(T)\\ \mathbf n_{r}=n}}\prod_{\beta\in T^{\infty}}\|v_{n_\beta}\|_{2}^{q}\Big)^{\frac1{q}}.
\end{equation}
Applying Young's inequality for the last summand implies the desired estimate. For the operator $N_{r}^{(J+1)}$ the proof is the same but in addition we use Lemma \ref{lem} for the operator $R_{2}^{t}-R_{1}^{t}$. 
\end{proof}

\begin{remark}
\label{important94536}
Note that there is an extra factor $\sim J$ when we estimate the differences $N_{0}^{(J+1)}(v)-N_{0}^{(J+1)}(w)$ since $|a^{2J+1}-b^{2J+1}|\lesssim(\sum_{j=1}^{2J+1}a^{2J+1-j}b^{j-1})|a-b|$ has $O(J)$ many terms. Also, we have $c_{J}=|\mathcal I(J)|$ many summands in the operator $N_{0}^{(J+1)}$ since there are $c_{J}$ many trees of the $J$th generation and $c_{J}$ behaves like a double factorial in $J$ (see (\ref{setsetset3})). However, these extra terms do not cause any problem since the constant we obtain from \eqref{wow4352} decays like a fractional power of a double factorial in $J$, or to be more precise we have 
\begin{equation}
\label{factori}
\frac{C^{\frac{J}2}\ c_{J}}{\prod_{j=2}^{J}(2j+3)^{\frac32}}\sim\frac1{J^\frac{J}2}.
\end{equation}
\end{remark}

Here is the estimate of the operator $N_{1}^{(J+1)}$.

\begin{lemma}
\label{finaal2}
$$\|N_{1}^{(J+1)}(v)\|_{l^{q}L^{2}}\lesssim N^{-\frac{J-1}2(1-\frac1{100})}\|v\|_{M_{2,q}}^{2J+3},$$
and
$$\|N_{1}^{(J+1)}(v)-N_{1}^{(J+1)}(w)\|_{l^{q}L^{2}}\lesssim N^{-\frac{J-1}2(1-\frac1{100})}(\|v\|_{M_{2,q}}^{2J+2}+\|w\|_{M_{2,q}}^{2J+2})\|v-w\|_{M_{2,q}}.$$
\end{lemma}
\begin{proof}
Since we are on $C_{J}$ the requirement $|\tilde{\phi}_{J+1}|=|\tilde{\phi}_{J}+\phi_{J+1}|\lesssim(2J+3)^{3}|\tilde{\phi}_{J}|^{1-\frac1{100}}$ (similar considerations for the requirement $|\tilde{\phi}_{J+1}|\lesssim(2J+3)^{3}|\phi_{1}|^{1-\frac1{100}}$) implies that $|\phi_{J+1}|\lesssim J^{3}|\tilde{\phi}_{J}|$. With the use of \eqref{godhm} we obtain
\begin{equation}
\label{godhm1}
|\phi_{1}||\phi_{J+1}|\prod_{j=2}^{J}(2j+3)^{3}N^{1-\frac1{100}}|\phi_{j}|\lesssim J^{3}\prod_{j=1}^{J}|\tilde{\phi}_{j}|^{2}.
\end{equation}
Following the argument in \eqref{nvbt6} and \eqref{bvv89} with the use of \eqref{er5376} we arrive at 
\begin{equation}
\label{godhm2}
\sum_{\substack{\mathbf n\in\mathcal R(T)\\ \mathbf n_{r}=n}}\prod_{j=1}^{J}\frac1{|\tilde{\phi}_{j}|^{2}}\lesssim\frac{N^{-(J-1)(1-\frac1{100})}}{\prod_{j=2}^{J-1}(2j+3)^{3}}\sum_{\substack{\mathbf n\in\mathcal R(T)\\ \mathbf n_{r}=n}}\prod_{j=1}^{J+1}\frac1{|\phi_{j}|}\lesssim\frac{C^{J+1}\ N^{-(J-1)(1-\frac1{100})}}{\prod_{j=2}^{J-1}(2j+3)^{3}}
\end{equation}
which finishes the proof since trivially
$$\sum_{\substack{\mathbf n\in\mathcal R(T)\\ \mathbf n_{r}=n}}\|\tilde q^{J,t}_{T^0,\mathbf n}(N_{1}^{t,\alpha}(\{v_{n_{\beta}}\}_{\beta\in T^{\infty}}))\|_{2}\lesssim$$
$$\sum_{\substack{\mathbf n\in\mathcal R(T)\\ \mathbf n_{r}=n}}\Big(\|v_{n_{\alpha_{1}}}\|_{2}\|v_{n_{\alpha_{2}}}\|_{2}\|v_{n_{\alpha_{3}}}\|_{2}\prod_{\beta\in T^{\infty}\setminus\{\alpha\}}\|v_{n_\beta}\|_{2}\Big)\Big(\prod_{k=1}^{J}\frac1{|\tilde\phi_{k}|}\Big)$$
$$\Big(\sum_{\substack{|\phi_{1}|>N\\ |\tilde\phi_{j}|>(2j+1)^{3}N^{1-\frac1{100}}\\ j=2,\ldots,J}}\prod_{j=1}^{J}\frac1{|\tilde{\phi}_{j}|^{q'}}\Big)^{\frac1{q'}}
\Big(\sum_{\substack{\mathbf n\in\mathcal R(T)\\ \mathbf n_{r}=n}}\|v_{n_{\alpha_{1}}}\|_{2}^{q}\|v_{n_{\alpha_{2}}}\|_{2}^{q}\|v_{n_{\alpha_{3}}}\|_{2}^{q}\prod_{\beta\in T^{\infty}\setminus\{\alpha\}}\|v_{n_{\beta}}\|_{2}^{q}\Big)^{\frac1{q}}.$$
The first summand is controlled by \eqref{godhm2} and for the second summand we apply Young's inequality.
\end{proof}

\begin{remark}
\label{reme}
For $s>0$ we have to observe that all previous lemmata hold true if we replace the $l^{q}L^{2}$ norm by the $l^{q}_{s}L^{2}$ norm and the $M_{2,q}(\R)$ norm by the $M_{2,q}^{s}(\R)$ norm. To this end notice that for large $n^{(j)}$ there exists at least one of $n_{1}^{(j)},n_{2}^{(j)},n_{3}^{(j)}$ such that $|n_{k}^{(j)}|\geq\frac13|n^{(j)}|$, $k\in\{1,2,3\}$, since we have the relation $n^{(j)}\approx n_{1}^{(j)}-n_{2}^{(j)}+n_{3}^{(j)}$. Thus, in the estimates of the $J$th generation, there exists at least one frequency $n_{k}^{(j)}$ for some $j\in\{1,\ldots,J\}$ with the property
$$\langle n\rangle^{s}\leq 3^{js}\langle n_{k}^{(j)}\rangle^{s}\leq 3^{Js}\langle n_{k}^{(j)}\rangle ^{s}.$$
This exponential growth does not affect our calculations due to the double factorial growth in the denominator of (\ref{factori}).
\end{remark}

\end{section}

\begin{section}{existence and uniqueness}
\label{cbgg342}
The calculations of this section are the same as the ones given in \cite[Subsections 2.3 and 2.4]{NP}. The only small difference is the following lemma which deals with the behaviour of the remainder operator $N_{2}^{J}$ as $J\to\infty$.

\begin{lemma}
\label{finafinafina}
For $v\in M_{2,q}^{s}(\R)$, with $s\geq0$ and $q\in[1,2]$, if $M_{2,q}^{s}(\R)\hookrightarrow L^{3}(\R)$ then we have
$$\lim_{J\to\infty}\|N_{2}^{(J)}(v)\|_{l^{\infty}L^{2}}=0.$$
\end{lemma}
\begin{proof}
By \eqref{fina1} we can write the remainder operator as the following sum
\begin{equation}
\label{ppp}
N_{2}^{(J)}(v)(n)=\partial_{t}(N_{0}^{(J+1)}(v)(n))+\sum_{T\in T(J)}\sum_{\alpha\in T^{\infty}}\sum_{\substack{\mathbf n\in\mathcal R(T)\\ \mathbf n_{r}=n}}\tilde q^{J,t}_{T^0,\mathbf n}(\partial_{t}^{(\alpha)}(\{v_{n_{\beta}}\}_{\beta\in T^{\infty}})),
\end{equation}
where we define the action of $\partial_{t}^{(\alpha)}$ onto the set of functions $\{v_{n_{\beta}}\}_{\beta\in T^{\infty}}$ to be the same set of functions except for the $\alpha$ node where we replace $v_{n_{\alpha}}$ by the function $\partial_{t}v_{n_{\alpha}}$.

We control the first summand $\partial_{t}(N_{0}^{(J+1)}(v)(n))$ by Lemma \ref{finaal}. For the last summand of the RHS of (\ref{ppp}) we estimate its $L^{2}$ norm as follows
$$\sum_{T\in T(J)}\sum_{\alpha\in T^{\infty}}\sum_{\substack{\mathbf n\in\mathcal R(T)\\ \mathbf n_{r}=n}}\|\tilde{q}^{J,t}_{T^{0}}(N_{1}^{t,\alpha}(\{w_{n_{\beta}}\}_{\beta\in T^{\infty}}))\|_{2}\lesssim$$ 
 $$\sum_{T\in T(J)}\sum_{\alpha\in T^{\infty}}\sum_{\substack{\mathbf n\in\mathcal R(T)\\ \mathbf n_{r}=n}}\prod_{\beta\in T^{\infty}\setminus\{\alpha\}}\|v_{n_{\beta}}\|_{2}\ \frac{\|\partial_{t}v_{n_{\alpha}}\|_{2}}{\prod_{k=1}^{J}|\tilde{\phi}_{k}|},$$ 
 which by H\"older's inequality with exponents $\frac1{q}+\frac1{q'}=1$ and \eqref{factori} implies the upper bound
$$\frac1{J^{\frac{J}2}}\ \sum_{T\in T(J)}\sum_{\alpha\in T^{\infty}}\Big(\sum_{\substack{\mathbf n\in\mathcal R(T)\\ \mathbf n_{r}=n}}\prod_{\beta\in T^{\infty}\setminus\{\alpha\}}\|v_{n_{\beta}}\|_{2}^{q}\|\partial_{t}v_{n_{\alpha}}\|_{2}^{q}\Big)^{\frac1{q}}.$$
Then for the sum inside the parenthesis we apply Young's inequality in the discrete variable where for the first $2J$ functions we take the $l^{1}$ norm and for the last the $l^{\infty}$ norm we arrive at the estimate 
$$\|v\|_{M_{2,q}}^{2J}\sup_{n\in\Z}\|\partial_{t}v_{n}\|_{2}=\|v\|_{M_{2,q}}^{2J}\|\partial_{t}v_{n}\|_{l^{\infty}L^{2}}.$$
Since by (\ref{main3}) we have $\partial_{t}v_{n}=e^{it\partial_{x}^{4}}\Box_{n}(|u|^{2}u)$ it is straightforward to obtain
$$\|\partial_{t}v_{n}\|_{l^{\infty}L^{2}}\lesssim\|v\|_{M_{2,q}}^{3}.$$
Indeed, from (\ref{Sch}) and since $\Box_{n}(|u|^{2}u)$ is nicely localised it suffices to estimate 
$$\|\Box_{n}(|u|^{2}u)\|_{2}\lesssim\|\Box_{n}(|u|^{2}u)\|_{1}\lesssim\||u|^{2}u\|_{1}=\|u\|_{3}^{3}\lesssim\|u\|_{M_{2,q}^{s}}^{3}=\|v\|_{M_{2,q}^{s}}^{3},$$ 
where we used \eqref{Bern1}, Lemma \ref{Bern} and the embedding $M^{s}_{2,q}(\R)\hookrightarrow L^{3}(\R)$. Therefore, putting everything together we arrive at
$$\|N_{2}^{(J)}(v)\|_{l^{\infty}L^{2}}\lesssim\frac1{J^{\frac{J}2}}\|v\|_{M_{2,q}^{s}}^{2J+3},$$
which finishes the proof. 
\end{proof}

Observe that the assumption $M_{2,q}^{s}(\R)\hookrightarrow L^{3}(\R)$ implies that if $u$ is a solution of the biharmonic NLS \eqref{maineq} in the space $C([0,T],M_{2,q}^{s}(\R))$ then $u$ and hence $v=e^{it\partial_{x}^{4}}u$ are elements of $X_{T}\hookrightarrow C([0,T], L^{3}(\R)).$ Thus, the nonlinearity of the biharmonic NLS \eqref{maineq} makes sense as an element of $C([0,T], L^{1}(\R))$ and by \eqref{main3} we obtain that $\partial_{t}v_{n}\in C([0,T],L^{1}(\R))$. The next lemma justifies all the formal calculations that were performed in the previous sections (for a proof see e.g. \cite[Lemma 27]{NP}).

\begin{lemma}
\label{didi}
Let $f,\partial_{t}f\in C([0,T],L^{1}(\R^{d}))$ and define the distribution $\int_{\R^{d}}f(\cdot,x)dx$ by
$$\Big\langle \int_{\R^{d}}f(\cdot, x)dx, \phi\Big\rangle=\int_{\R}\int_{\R^{d}}f(t,x)\phi(t)dxdt,$$
for $\phi\in C^{\infty}_{c}(\R).$ Then, $\partial_{t}\int_{\R^{d}}f(\cdot,x)dx=\int_{\R^{d}}\partial_{t}f(\cdot,x)dx.$
\end{lemma}
Here is an application of the lemma. Consider \eqref{ttr} for fixed $n$ and $\xi.$ We want to apply Lemma \ref{didi} to the function
$$f(t,\xi_{1},\xi_{3})=\sigma_{n}(\xi)\ \frac{e^{it\Phi_{4}(\xi,\xi_{1},\xi-\xi_{1}-\xi_{3},\xi_{3})}}{i\Phi_{4}(\xi,\xi_{1},\xi-\xi_{1}-\xi_{3},\xi_{3})}\ \hat{v}_{n_{1}}(\xi_{1})\hat{\bar{v}}_{n_{2}}(\xi-\xi_{1}-\xi_{3})\hat{v}_{n_{3}}(\xi_{3}),$$
where $\xi\approx n, \xi_{1}\approx n_{1},\xi_{3}\approx n_{3}, \xi-\xi_{1}-\xi_{3}\approx -n_{2}$ and $(n,n_{1},n_{2},n_{3})\in A_{N}(n)^{c}$ given by \eqref{idid}. Notice that $f, \partial_{t}f \in C([0,T],L^{1}(\R^{2}))$ since $v\in C([0,T],M_{2,q}^{s}(\R))$ and $\partial_{t}v_{n}\in C([0,T],L^{1}(\R))$ for all integers $n$. Thus,
$$\partial_{t}\Big[\int_{\R^2}\sigma_{n}(\xi)\ \frac{e^{it\Phi_{4}(\xi,\xi_{1},\xi-\xi_{1}-\xi_{3},\xi_{3})}}{i\Phi_{4}(\xi,\xi_{1},\xi-\xi_{1}-\xi_{3},\xi_{3})}\ \hat{v}_{n_{1}}(\xi_{1})\hat{\bar{v}}_{n_{2}}(\xi-\xi_{1}-\xi_{3})\hat{v}_{n_{3}}(\xi_{3})d\xi_{1}d\xi_{3}\Big]=$$
$$\int_{\R^2}\sigma_{n}(\xi)\partial_{t}\Big[\frac{e^{it\Phi_{4}(\xi,\xi_{1},\xi-\xi_{1}-\xi_{3},\xi_{3})}}{i\Phi_{4}(\xi,\xi_{1},\xi-\xi_{1}-\xi_{3},\xi_{3})}\ \hat{v}_{n_{1}}(\xi_{1})\hat{\bar{v}}_{n_{2}}(\xi-\xi_{1}-\xi_{3})\hat{v}_{n_{3}}(\xi_{3})\Big]d\xi_{1}d\xi_{3}=$$
$$\int_{\R^2}\sigma_{n}(\xi)\partial_{t}\Big[\frac{e^{it\Phi_{4}(\xi,\xi_{1},\xi-\xi_{1}-\xi_{3},\xi_{3})}}{i\Phi_{4}(\xi,\xi_{1},\xi-\xi_{1}-\xi_{3},\xi_{3})}\Big]\hat{v}_{n_{1}}(\xi_{1})\hat{\bar{v}}_{n_{2}}(\xi-\xi_{1}-\xi_{3})\hat{v}_{n_{3}}(\xi_{3})d\xi_{1}d\xi_{3}+$$
$$\int_{\R^2}\sigma_{n}(\xi)\frac{e^{it\Phi_{4}(\xi,\xi_{1},\xi-\xi_{1}-\xi_{3},\xi_{3})}}{i\Phi_{4}(\xi,\xi_{1},\xi-\xi_{1}-\xi_{3},\xi_{3})}\partial_{t}\Big[\hat{v}_{n_{1}}(\xi_{1})\hat{\bar{v}}_{n_{2}}(\xi-\xi_{1}-\xi_{3})\hat{v}_{n_{3}}(\xi_{3})\Big]d\xi_{1}d\xi_{3}.$$
In the second equality we used the product rule which is applicable since $v\in C([0,T],L^{3}(\R))$ implies that $\partial_{t}v_{n}\in C([0,T],L^{1}(\R)).$

Finally it remains to justify the interchange of differentiation in time and summation in the discrete variable but this is done in exactly the same way as in \cite[Lemma 5.1]{GKO}. Similar arguments justify the interchange on the $J$th step of the infinite iteration procedure. 

Having proved these lemmata we define the partial sum operators $\Gamma_{v_{0}}^{(J)}$ as
\begin{equation}
\label{gamaa}
\Gamma_{v_{0}}^{(J)}v(t)=v_{0}+\sum_{j=2}^{J}N_{0}^{(j)}(v)(n)-\sum_{j=2}^{J}N_{0}^{(j)}(v_{0})(n)
\end{equation}
$$+\int_{0}^{t}R_{1}^{\tau}(v)(n)+R_{2}^{\tau}(v)(n)+\sum_{j=2}^{J}N_{r}^{(j)}(v)(n)+\sum_{j=1}^{J}N_{1}^{(j)}(v)(n)\ d\tau,$$
where we have $N_{1}^{(1)}:=N_{11}^{t}$ from \eqref{main13}, $N_{0}^{(2)}:=N_{21}^{t}$ from \eqref{nex}, $N_{1}^{(2)}:=N_{31}^{t}$ from \eqref{patel} and $N_{r}^{(2)}:=N_{4}^{t}$ from \eqref{patel2} and $v_{0}\in M_{2,q}(\R)$ is a fixed function. 

The argument from \cite[Subsection 2.3]{NP} shows that for sufficiently large $N$ and sufficiently small $T>0$ these operators $\Gamma^{(J)}_{v_{0}}$ are well defined in $X_{T}\coloneqq C([0,T],M_{2,q}(\R))$ for every $J\in\mathbb N\cup\{\infty\}$. We write $\Gamma_{v_{0}}$ for $\Gamma_{v_{0}}^{\infty}$. 

The differentiation by parts argument presented in the previous sections shows that if the function $v$ is sufficiently smooth or if the modulation space $M_{2,q}^{s}(\R)$ embeds in $L^{3}(\R)$ then a solution $v$ of the biharmonic NLS with initial data $v_{0}$ is a fixed point of the operator $\Gamma_{v_{0}}$, i.e.
\begin{equation}
\label{argg7}
v(t)=v_{0}+i\int_{0}^{t}N_{1}^{\tau}(v)-R_{1}^{\tau}(v)+R_{2}^{\tau}(v)\ d\tau=
\end{equation}
$$v_{0}+\sum_{j=2}^{\infty}N_{0}^{(j)}(v)(n)-\sum_{j=2}^{\infty}N_{0}^{(j)}(v_{0})(n)$$
$$+\int_{0}^{t}R_{1}^{\tau}(v)(n)+R_{2}^{\tau}(v)(n)+\sum_{j=2}^{\infty}N_{r}^{(j)}(v)(n)+\sum_{j=1}^{\infty}N_{1}^{(j)}(v)(n)\ d\tau=\Gamma_{v_{0}}v.$$

The important property of the $\Gamma_{v_{0}}$ operators is that if we are given two initial data $v_{0}^{(1)}$ and $v_{0}^{(2)}$ that are close in $M_{2,q}^{s}(\R)$ then if $v^{(1)}$ is the solution to the biharmonic NLS with initial data $v_{0}^{(1)}$ and $v^{(2)}$ is the solution with initial data $v_{0}^{(2)}$ we have
\begin{equation}
\label{done4222}
\|v^{(1)}-v^{(2)}\|_{X_{T}}=\|\Gamma_{v_{0}^{(1)}}v^{(1)}-\Gamma_{v_{0}^{(2)}}v^{(2)}\|_{X_{T}}\lesssim\|v_{0}^{(1)}-v_{0}^{(2)}\|_{M_{2,q}^{s}}.
\end{equation}

As it was done in \cite{CHKP1} and \cite{NP} for the cubic NLS (see also \cite{CHKP2}), the proof of Theorem \ref{th1} consists of approximating the initial data $v_{0}\coloneqq u_{0}$ by smooth functions $\{v_{0}^{(m)}\}_{m\in\mathbb N}$ in $M_{2,q}^{s}(\R)$, solving the biharmonic NLS for such $v_{0}^{(m)}$ in $X_{T}$, using the $\Gamma_{v_{0}^{(m)}}$ operators, with a smooth solution $v^{(m)}$, showing that $v^{(m)}$ have a common time of existence for all $m\in\mathbb N$ and that the sequence $\{v^{(m)}\}_{m\in\mathbb N}$ is Cauchy in $X_{T}$. The limit $v$ is the weak solution in the extended sense of the biharmonic NLS with initial data $v_{0}$ that we were trying to find. The nonlinearity $\mathcal N(u)$ for $u=e^{-it\partial_{x}^{4}}v$ is equal to $\lim_{n\to\infty}\mathcal N(u^{(m)})$ in the sense of distributions in $(0,T)\times\R$ where $u^{(m)}=e^{-it\partial_{x}^{4}}v^{(m)}$.

The proof of Theorem \ref{mainyeah} follows from \eqref{done4222} since if there are two solutions $u_{1}$ and $u_{2}$ with the same initial datum $u_{0}$ we obtain 
$$\|u_{1}-u_{2}\|_{X_{T}}=\|\Gamma_{u_{0}}u_{1}-\Gamma_{u_{0}}u_{2}\|_{X_{T}}\lesssim\|u_{0}-u_{0}\|_{M_{2,q}^{s}}=0.$$
 
\end{section}

\begin{section}{the general higher order nonlinear schr\"odinger equation}
\label{higherNLS478}

For $k\in\mathbb Z_{+}$ consider the following Cauchy problem of the higher order NLS
\begin{equation}
\label{highergenNLS}
\begin{cases} i\partial_{t}u-(-1)^{k}\partial_{x}^{2k}u\pm|u|^{2}u=0\\
u(0,x)=u_{0}(x).\\
\end{cases}
\end{equation}
It is straightforward to see that its phase factor is given by the function $\Phi_{2k}$ defined in \eqref{aaa23} which enjoys the following factorization.

\begin{proposition}
\label{factoriPhi375}
Under the assumption $\xi=\xi_{1}-\xi_{2}+\xi_{3}$ we have that
\begin{equation}
\label{polyharm}
\Phi_{2k}(\xi,\xi_{1},\xi_{2},\xi_{3})=(\xi_{1}-\xi_{2})(\xi_{3}-\xi_{2})P_{k}(\xi_{1},\xi_{2},\xi_{3})
\end{equation}
where $P_{k}(\xi_{1},\xi_{2},\xi_{3})\in\mathbb Z[\xi_{1},\xi_{2},\xi_{3}]$ is a non-negative homogeneous polynomial of degree $2k-2$ with only the trivial root. More precisely, we have the formula
\begin{equation}
\label{polyharmpolyn}
P_{k}(\xi_{1},\xi_{2},\xi_{3})=\sum_{m=1}^{2k-2}\xi_{3}^{m}\sum_{q=0}^{2k-2-m}\xi_{1}^{q}\xi_{2}^{2k-2-m-q}+\sum_{q=0}^{2k-2}\xi_{1}^{q}\xi_{2}^{2k-2-q}+ 
\end{equation}
$$\sum_{m=1}^{2k-1}\xi_{1}^{2k-1-m}\sum_{p=0}^{m-1}{m \choose p}(\xi_{1}-\xi_{2})^{m-p-1}\xi_{3}^{p}.$$
As a consequence, if $\xi=\xi_{1}-\xi_{2}+\xi_{3}$ then
\begin{equation}
\label{impoff421}
\left|\Phi_{2k}\right| \sim \max \left\{|\xi|,\left|\xi_{1}\right|,\left|\xi_{2}\right|,\left|\xi_{3}\right|\right\}^{2k-2}\left|\xi-\xi_{1} \| \xi-\xi_{3}\right|\gtrsim \left|\Phi_{2}\right|^{k}
\end{equation}
for all $k\in\mathbb N$.
\end{proposition}

\begin{proof}
As $\xi^{2}-\xi_{1}^{2}+\xi_{2}^{2}-\xi_{3}^{2} = 2(\xi_1-\xi_2)(\xi_3-\xi_2)$, we may write
\begin{equation}
\frac{1}{2}\ P_{k}\left(\xi_{1}, \xi_{2}, \xi_{3}\right) = \frac{\xi^{2 k}-\xi_{1}^{2 k}+\xi_{2}^{2 k}-\xi_{3}^{2 k}}{\xi^{2}-\xi_{1}^{2}+\xi_{2}^{2}-\xi_{3}^{2}}= 
\end{equation}
$$s^{2k-2}\ \frac{y^{2 k}+(1-y)^{2k}-(x^{2k} + (1-x)^{2k})}{y^{2}+(1-y)^{2}-(x^{2} + (1-x)^{2})}.$$
where we set $s = \xi + \xi_2 = \xi_1 + \xi_3$, $x = \xi_1/s$ and $y = \xi_2/s$. Without loss of generality $s \neq 0$, else we see directly from the above expression that $P_k$ can only be zero when $\xi_1 = \xi_2 = \xi_3 = 0$. Consider the function $f(t) = t^{2k} + (1-t)^{2k}$. Then $f(t)$ is symmetric around $t = \frac{1}{2}$ and a change of coordinate $u = t - \frac{1}{2}$ shows that
\begin{equation}
f(t(u)) = \left(u+\frac{1}{2}\right)^{2k} + \left(u-\frac{1}{2}\right)^{2k} = \sum_{l=0}^{2k} {2k \choose 2l}u^{2(k-l)}2^{-2l+1}.
\end{equation}
This implies that $f$ is convex and strictly decreasing (respectively increasing) when $t < \frac{1}{2}$ (respectively $t > \frac{1}{2}$). Hence, $f(x) = f(y)$ can only hold if $x = y$ or $x + y = 1$, which shows that $\xi_3 = \xi_2$, respectively $\xi_1 = \xi_2$. Using l'Hospital's rule, one easily sees that in either case $P_k \neq 0$ for nontrivial $(\xi_1,\xi_2,\xi_3)$ whenever $k \geq 2$.

In order to prove \eqref{impoff421}, note that both $\Phi_{2k}$ and $P_k$ are homogenous polynomials, and therefore, it is enough to show the required estimate on the unit circle. Also note that, since there is at least one positive value (choose $(1,-1,1)$ for example), we have $P_k > 0$ on $\R^3 \setminus \{0\}$. Notice that $P_k \lesssim \max(|\xi_1|,|\xi_2|,|\xi_3|)^{2k-2}$ is true for any polynomial $P_{k}$ of degree $2k-2$. The other direction in this estimate then holds because $P_k$ attains a strictly positive minimum on the sphere. To conclude the proof observe that $\Phi_2^{k-1} \lesssim \max(|\xi_1|,|\xi_2|,|\xi_3|)^{2k-2}$ since $\Phi_2^{k-1}$ is a polynomial of degree $2k-2$.
\end{proof}

\begin{remark}
In the case where $k=1$ the polynomial is $P_{1}(\xi_{1},\xi_{2},\xi_{3})=2$, whereas in the case $k=2$ the polynomial is
$$P_{2}(\xi_{1},\xi_{2},\xi_{3})=\xi_{1}^{2}+\xi_{2}^{2}+\xi_{3}^{2}+(\xi_{1}-\xi_{2}+\xi_{3})^{2}+2(\xi_{1}+\xi_{3})^{2},$$ 
which is proved in \cite[Lemma 3.1]{OT}. Using \eqref{polyharmpolyn} for the case $k=3$ after some tedious but elementary calculations we obtain the following expression for $P_{3}$:
\begin{equation}
\label{factpolynharm4}
P_{3}(\xi_{1},\xi_{2},\xi_{3})=\xi_{2}^{4}+(\xi_{1}-\xi_{2}+\xi_{3})^{4}+\frac12(\xi_{1}^{2}+\xi_{3}^{2})\Big[(\xi_{1}-\xi_{2}+\xi_{3})^{2}+\xi_{1}^{2}+\xi_{2}^{2}+\xi_{3}^{2}\Big]+
\end{equation}
$$2(\xi_{1}+\xi_{3})^{2}\Big[(\xi_{1}-\xi_{2}+\xi_{3})^{2}+\xi_{2}^{2}\Big]+2(\xi_{1}+\xi_{3})(\xi_{1}^{3}+\xi_{3}^{3}).$$
Notice that all summands are positive since all of them are even powers except the very last summand which is also positive but for a different reason: $\xi_{1}+\xi_{3}$ and $\xi_{1}^{3}+\xi_{3}^{3}$ have the same sign. Thus, for the polynomial $P_{3}(\xi_{1},\xi_{2},\xi_{3})$ the trivial root $(0,0,0)$ is the only real root. 
\end{remark}

Because of \eqref{impoff421}, it is evident that all the calculations presented in Sections \ref{firststeps}, \ref{treesandinduction39} and \ref{cbgg342} go through for the general higher order NLS \eqref{highergenNLS}, that is Theorems \ref{th1} and \ref{mainyeah} (and Remarks \ref{explncases}, \ref{unconposed}) hold true. The same reasoning applies to the mixed order equation \eqref{higheryes} described in Remark \ref{sixthhigher23} since the phase factors sum up as
\begin{equation}
\label{sumupall7399}
\sum_{j=1}^{M}\epsilon_{j}\Phi_{2j}(\xi,\xi_{1},\xi_{2},\xi_{3})=(\xi_{1}-\xi_{2})(\xi_{1}-\xi_{3})\Big(\sum_{j=1}^{M}\epsilon_{j}P_{j}(\xi_{1},\xi_{2},\xi_{3})\Big)
\end{equation}
whenever $\xi=\xi_{1}-\xi_{2}+\xi_{3}$.

\textbf{Acknowledgments}: Funded by the Deutsche Forschungsgemeinschaft (DFG, German Research
Foundation) - Project-ID 258734477 - SFB 1173.

\end{section}

\end{document}